\documentclass[11pt,letterpaper]{amsart}

\usepackage[margin=1.08in]{geometry}
\usepackage[T1]{fontenc}
\usepackage[utf8]{inputenc}
\usepackage{amsmath,amssymb,amsthm,mathtools}
\usepackage{microtype}
\usepackage{xcolor}
\IfFileExists{libertinus.sty}{\usepackage{libertinus}}{\usepackage{lmodern}}
\usepackage{enumitem}
\usepackage{needspace}
\usepackage{cite}
\usepackage{hyperref}
\usepackage{aliascnt}
\usepackage{tikz}
\usetikzlibrary{arrows.meta,positioning,calc}
\usepackage[nameinlink,capitalize]{cleveref}

\hypersetup{
  colorlinks=true,
  linkcolor=black,
  citecolor=black,
  urlcolor=black,
  pdftitle={Hardy--Littlewood Constants Beyond Polarization},
  pdfauthor={Daniel M. Pellegrino and Eduardo V. Teixeira}
}
\setlist{itemsep=0.25em,topsep=0.45em}
\numberwithin{equation}{section}
\allowdisplaybreaks[1]
\newcommand{\revision}[1]{#1}
\newcommand{\auditfix}[1]{#1}
\definecolor{auditpurple}{RGB}{0,0,0}
\newcommand{\minorfix}[1]{{\color{auditpurple}%
  \hypersetup{linkcolor=black,citecolor=black,urlcolor=black}#1}}

\newcommand{\C}{\mathbb C}

\newcommand{\N}{\mathbb N}
\newcommand{\T}{\mathbb T}
\newcommand{\Pol}{\mathcal P}
\newcommand{\Acal}{\mathcal A}
\newcommand{\e}{\mathrm e}
\DeclareMathOperator{\card}{card}

\newtheorem{theorem}{Theorem}[section]
\newaliascnt{proposition}{theorem}
\newtheorem{proposition}[proposition]{Proposition}
\aliascntresetthe{proposition}
\newaliascnt{lemma}{theorem}
\newtheorem{lemma}[lemma]{Lemma}
\aliascntresetthe{lemma}
\newaliascnt{corollary}{theorem}

\aliascntresetthe{corollary}
\theoremstyle{remark}
\newtheorem{remark}[theorem]{Remark}
\theoremstyle{plain}
\newtheorem{theoremA}{Theorem}
\newtheorem{theoremB}{Theorem}
\newtheorem{theoremC}{Theorem}

\crefname{theorem}{theorem}{theorems}
\Crefname{theorem}{Theorem}{Theorems}
\crefname{proposition}{proposition}{propositions}
\Crefname{proposition}{Proposition}{Propositions}
\crefname{lemma}{lemma}{lemmas}
\Crefname{lemma}{Lemma}{Lemmas}
\crefname{corollary}{corollary}{corollaries}
\Crefname{corollary}{Corollary}{Corollaries}

\title[Hardy--Littlewood constants beyond polarization]
{Hardy--Littlewood Constants Beyond Polarization}

\author[D. M. Pellegrino]{Daniel M. Pellegrino}
\address{Departamento de Matem\'atica, Universidade Federal da Para\'iba, Jo\~ao Pessoa, PB, Brazil}
\email{daniel.pellegrino@academico.ufpb.br}

\author[E. V. Teixeira]{Eduardo V. Teixeira}
\address{Department of Mathematics, Oklahoma State University, Stillwater, OK 74078, USA}
\email{eduardo.teixeira@okstate.edu}

\date{}

\begin{document}

\begin{abstract}
{
\auditfix{The polynomial Hardy--Littlewood inequalities bound coefficient
norms of homogeneous polynomials on $\ell_p$ balls by their supremum norms,
with constants independent of the dimension. Their $p=\infty$ endpoint is
the Bohnenblust--Hille inequality.} A longstanding obstruction
to connecting their constants is the polarization loss in the usual
multilinear approach. We overcome this obstruction by proving
\revision{${D_{m,p_m}}\le(1+o(1))D_m$ whenever $p_m/m^2\to\infty$},
where ${D_{m,p}}$ and $D_m$ are the optimal complex
polynomial Hardy--Littlewood and Bohnenblust--Hille constants. {At the opposite boundary, we determine the optimal endpoint constant exactly,
${D_{m,m}}=m$. We then prove the full first-order asymptotic
${D_{m,p_m}}\sim m$ whenever $p_m\ge m$ and $p_m-m\to0$, and the exact
polynomial growth ${D_{m,p_m}}=m^{1+o(1)}$ on the wider scale
$p_m-m=o((m/\log m)^{1/3})$.}
}
\end{abstract}

\maketitle
\enlargethispage{3pt}
\tableofcontents
\part{Introduction and proof mechanisms}
\section{Introduction}\label{sec:introduction}

The Bohnenblust--Hille inequality arose from the study of absolute convergence of
Dirichlet series.  {In 1931 Bohnenblust and Hille proved that the
$\ell_{2m/(m+1)}$-norm of the coefficients of an $m$-homogeneous polynomial
is bounded by its supremum norm on the unit ball of $\ell_\infty$, multiplied
by a constant depending only on $m$ \cite{BohnenblustHille1931}.}  The
quantitative problem of understanding the dependence of the optimal constant on
$m$ has since become a central part of the theory.  In the complex polynomial
setting, hypercontractive bounds were obtained in \cite{DefantEtAl2011}, and the
subexponential regime was reached in \cite{BayartPellegrinoSeoane2014}.
{In \cite{PellegrinoTeixeiraBH2026}, we obtained an upper bound of
square-root scale, up to a subpower factor, for the optimal complex polynomial
Bohnenblust--Hille constants.}

A second line of inequalities began with Hardy and Littlewood's 1934 theorem on
bilinear forms on $\ell_p$ spaces \cite{HardyLittlewood1934}.  The multilinear
extension was developed by Praciano-Pereira \cite{PracianoPereira1981}, and the
coefficient summability problem for homogeneous polynomials on $\ell_p$ spaces
was later studied systematically, in particular by Dimant and Sevilla-Peris
\cite{DimantSevillaPeris2016}; see also {\cite{AlbuquerqueBayartPellegrinoSeoane2016}} and \cite{AraujoEtAl2015} for the polynomial
Hardy--Littlewood inequalities in the form used here.
{The polynomial Hardy--Littlewood inequalities bound coefficient
norms by the supremum norm on the unit ball of $\ell_p$, with an exponent
depending on $m$ and $p$ and a constant independent of the number of variables.
At $p=\infty$, they reduce to the polynomial Bohnenblust--Hille inequality.}

{
Although the coefficient exponents agree at $p=\infty$, connecting the
constants has faced a longstanding obstruction. The usual passage from
multilinear Hardy--Littlewood inequalities to polynomial ones introduces a
polarization factor whose growth in $m$ is too large to recover the
Bohnenblust--Hille bounds. {To obtain our first result, we develop a coefficient-dependent diagonal-rescaling mechanism in which the associated entropy terms cancel, thereby avoiding the polarization loss.}

Let ${D_{m,p}}$ denote the optimal complex polynomial
Hardy--Littlewood constant. We prove that, as $m\to\infty$,
\[
 {D_{m,p_m}}\le (1+o(1))D_m
 \qquad\text{whenever}\qquad p_m/m^2\to\infty.
\]
Thus, every improvement in an upper bound for $D_m$ immediately gives an
improved Hardy--Littlewood bound, with asymptotically no loss in this range.
The argument also applies to polynomials in which each monomial involves at most a
fixed number of variables. For these polynomials, Bohnenblust--Hille bounds
tending to $1$ give Hardy--Littlewood bounds tending to $1$ whenever
$p_m/m^2\to\infty$; see Remark~\ref{rem:support-transfer}.
}

There is a second endpoint.  At $p=m$ the critical coefficient exponent is
$\infty$; no finite exponent gives a dimension-free inequality.  {Indeed, for
every finite $r$, the polynomial $P_n:\mathbb C^n\to\mathbb C$,
$P_n(z)=\sum_{j=1}^n z_j^m$, has $\|P_n\|_m=1$, while its coefficient vector
has $n$ entries equal to $1$, and hence has $\ell_r$ norm $n^{1/r}$.}  The endpoint inequality itself was studied in
Cavalcante and N\'u\~nez-Alarc\'on~\cite{CavalcanteNunez2016}, where the
$\ell_\infty$ coefficient regime was established and quantitative bounds for
the constant were obtained. Here we determine the optimal complex constant
exactly:
\[
 D_{m,m}=m.
\]
{Near this endpoint we then determine the growth more precisely.
If $p_m-m\to0$, then ${D_{m,p_m}}/m\to1$, giving the full first-order
asymptotic. On the wider scale $p_m-m=o((m/\log m)^{1/3})$, we still
determine the exact polynomial order: ${D_{m,p_m}}=m^{1+o(1)}$.}  In this sense the paper
identifies the quantitative behavior at both ends of the polynomial
Hardy--Littlewood range: the Bohnenblust--Hille endpoint $p=\infty$ and the
critical endpoint $p=m$.

For $m,n\in\mathbb N$, set
\[
 \mathcal M_m(n):=\{\alpha\in\mathbb N_0^n:|\alpha|=m\}.
\]
For $N\in\mathbb N$, let $\mathfrak m_N$ denote normalized Haar probability
measure on $\mathbb T^N$.
Let
\[
 P(z)=\sum_{|\alpha|=m}a_\alpha z^\alpha
\]
be an $m$-homogeneous polynomial on $\mathbb C^n$.  For $1\le p\le\infty$ write
\[
 \|P\|_p:=\sup_{\|z\|_p\le1}|P(z)|,
\]
and for the coefficient vector write $|P|_q=\|(a_\alpha)\|_q$, with the usual
maximum norm when $q=\infty$.  {We write $\Pol_m(\mathbb C^n)$ for the space of
$m$-homogeneous polynomials on $\mathbb C^n$.} For $m\le p\le\infty$ set
\begin{equation}\label{eq:HL-exponent}
 q(m,p):=
 \begin{cases}
 \infty,&p=m,\\[1mm]
 \dfrac{p}{p-m},&m<p<2m,\\[2mm]
 \dfrac{2mp}{mp+p-2m},&2m\le p<\infty,\\[2mm]
 \dfrac{2m}{m+1},&p=\infty.
 \end{cases}
\end{equation}
{For general Hardy--Littlewood type inequalities and the optimal exponent regimes, see also \cite{AlbuquerqueBayartPellegrinoSeoane2016}.}
The polynomial Hardy--Littlewood constant ${D_{m,p}}$ is
the least dimension-free constant in
\[
 |P|_{q(m,p)}\le {D_{m,p}}\|P\|_p.
\]
At $p=\infty$ we put
\begin{equation}\label{eq:BH-main-parameters}
 q_m:=\frac{2m}{m+1},
\end{equation}
{and set}
\[
 D_m={D_{m,\infty}},
\]
the optimal complex polynomial Bohnenblust--Hille constant.

{The upper endpoint is described by Theorem~\ref{thm:A}.}

\begin{theoremA}\label{thm:A}
Let $p_m>m$. If
\[
 \frac{p_m}{m^2}\longrightarrow\infty,
\]
then
\[
 {D_{m,p_m}}\le(1+o(1))D_m.
\]
\end{theoremA}

{The lower endpoint is first quantified on two scales. Theorem~\ref{thm:B} gives two quantitative estimates; part~\textup{(i)} will later be used in Theorem~\ref{thm:C}\textup{(iii)}.}

\begin{theoremB}\label{thm:B}
For every fixed $a>0$ the following hold.
\begin{enumerate}[label=\textup{(\roman*)}]
\item
\[
 \limsup_{m\to\infty}\sup_{m<p\le m+a(m/\log m)^{1/3}}
 \frac{\log {D_{m,p}}}{\log m}
 \le 1+{\frac{8}{3\pi\sqrt3}}a^{3/2}.
\]
\item
\[
 \limsup_{m\to\infty}\sup_{m<p\le m+a(m\log m)^{1/3}}
 \frac{\log {D_{m,p}}}{(\log m)^2}
 \le {\frac{8}{3\pi\sqrt3}}a^{3/2}.
\]
\end{enumerate}
\end{theoremB}

{At the critical endpoint we first determine the optimal constant exactly and
then quantify its stability. Theorem~\ref{thm:C} gives the endpoint identity,
the full first-order asymptotic when $p_m-m\to0$, and the exact polynomial
order on a wider neighborhood of $p=m$; part~\textup{(iii)} uses
Theorem~\ref{thm:B}\textup{(i)}.}

\Needspace{9\baselineskip}
\begin{theoremC}\label{thm:C}
The following statements hold.
\begin{enumerate}[label=\textup{(\roman*)}]
\item For every $m\ge2$,
\[
 {D_{m,m}}=m.
\]
\item If $p_m\ge m$ and $p_m-m\to0$, then
\[
 {D_{m,p_m}}=(1+o(1))m.
\]
\item If $p_m\ge m$ and $p_m-m=o((m/\log m)^{1/3})$, then
\[
 {D_{m,p_m}}=m^{1+o(1)}.
\]
\end{enumerate}
\end{theoremC}

\section{Proof mechanisms}\label{sec:proof-mechanisms}

{The upper and lower ranges are treated by different mechanisms. The proof of Theorem~\ref{thm:A} uses coefficient-dependent rescaling and entropy cancellation. {The two estimates in Theorem~\ref{thm:B} are obtained from scalar $H^1$ coefficient estimates and tensorization. Theorem~\ref{thm:C}\textup{(i)} follows from coefficient extraction on the torus, Theorem~\ref{thm:C}\textup{(ii)} uses multiplicity patterns and block projections, and Theorem~\ref{thm:C}\textup{(iii)} combines the resulting lower bound with Theorem~\ref{thm:B}\textup{(i)}.}}

\medskip
\noindent\textbf{Coefficient-dependent rescaling and entropy cancellation.}\par\smallskip
In the upper range $2m\le p<\infty$, write
\[
 \frac1q=\frac1{q_m}-\frac1p.
\]
{Then $q=q(m,p)$ is the Hardy--Littlewood coefficient exponent used throughout this mechanism.}

{To illustrate the mechanism, we first consider the $3$-homogeneous polynomial}
\[
 P:\mathbb C^2\longrightarrow\mathbb C,\qquad
 P(z_1,z_2)=a z_1^2z_2+b z_2^3,
\]
{with $(a,b)\ne(0,0)$, and set}
\[
 \mu:=\frac{|a|^q}{|a|^q+|b|^q}.
\]
The two monomials have exponent vectors $(2,1)$ and $(0,3)$, respectively. We weight the first vector by $\mu$ and the second by $1-\mu$. Their weighted average is
\[
 \mu(2,1)+(1-\mu)(0,3)=\bigl(2\mu,\,\mu+3(1-\mu)\bigr).
\]
Since every monomial has total degree $3$, we divide these two numbers by $3$ and set
\[
 t_1:=\frac{2\mu}{3},\qquad
 t_2:=\frac{\mu+3(1-\mu)}{3}=1-\frac{2\mu}{3}.
\]
In particular, $t_1,t_2\geq0$ and $t_1+t_2=1$.

{The rescaling acts on each pair $(z_1,z_2)\in\mathbb C^2$ with
$|z_1|,|z_2|\leq1$ by sending it to the pair
\[
 (z_1,z_2)\longmapsto
 (t_1^{1/p}z_1,t_2^{1/p}z_2).
\]
The first pair is an arbitrary point of the unit bidisc, and the second pair belongs
to the unit ball of $\ell_p^2$, because
\[
 |t_1^{1/p}z_1|^p+|t_2^{1/p}z_2|^p
 =t_1|z_1|^p+t_2|z_2|^p\leq t_1+t_2=1.
\]}
We therefore define
\[
 P_t:\mathbb C^2\longrightarrow\mathbb C,\qquad
 P_t(z_1,z_2):=P(t_1^{1/p}z_1,t_2^{1/p}z_2).
\]
{Thus, for every $(z_1,z_2)$ with $|z_1|,|z_2|\leq1$,}
\[
 |P_t(z_1,z_2)|\leq\|P\|_p,
 \qquad
 \|P_t\|_\infty\leq\|P\|_p.
\]

For a general $m$-homogeneous polynomial
\[
 P:\mathbb C^n\longrightarrow\mathbb C,\qquad
 P(z)=\sum_{|\alpha|=m}a_\alpha z^\alpha\neq0,
\]
we set
\[
 \mu_\alpha:=\frac{|a_\alpha|^q}{|P|_q^q}\quad(|\alpha|=m),\qquad
 t_i:=\frac1m\sum_{|\alpha|=m}\alpha_i\mu_\alpha\quad(1\leq i\leq n).
\]
Then $\mu_\alpha\geq0$, $t_i\geq0$,
\[
 \sum_{|\alpha|=m}\mu_\alpha=1,\qquad \sum_{i=1}^n t_i=1.
\]
The rescaled polynomial is
\[
 P_t:\mathbb C^n\longrightarrow\mathbb C,\qquad
 P_t(z_1,\ldots,z_n):=P(t_1^{1/p}z_1,\ldots,t_n^{1/p}z_n),
\]
{For $|z_i|\leq1$ $(1\leq i\leq n)$,
\[
 \sum_{i=1}^n|t_i^{1/p}z_i|^p\leq\sum_{i=1}^nt_i=1,
\]
so $\|P_t\|_\infty\leq\|P\|_p$.}

The entropy quantity introduced here will organize the two estimates proved in
Lemma~\ref{lem:coefficient-transfer} and used in the proof of Theorem~\ref{thm:A}.  {For a finite family
$w=(w_1,\ldots,w_N)$ satisfying $w_j\geq0$ and
$\sum_{j=1}^Nw_j=1$, define}
\begin{equation}\label{eq:Shannon-entropy}
 {H(w):=-\sum_{j=1}^Nw_j\log w_j,\qquad 0\log0:=0.}
\end{equation}
{This is the Shannon entropy of the weights $w_j$ \cite{Shannon1948}. It is
$0$ when all the weight is concentrated at one index and increases as the
weights are spread among several indices.}  We apply it to
$(\mu_\alpha)_{|\alpha|=m}$ and $(t_i)_{i=1}^n$, and set
\[
 I:=mH(t)-H(\mu).
\]

{In Section~\ref{sec:upper-HL}, the rescaling produces two estimates. {The first is the coefficient-norm estimate \eqref{eq:P_t-q0-lower}; the second is the derivative-sum estimate \eqref{eq:S-lower}.} For an $m$-homogeneous polynomial $Q(z)=\sum_{|\alpha|=m}b_\alpha z^\alpha$, set
\[
 S(Q):=\sum_{i=1}^n|\partial_iQ|_2
 =\sum_{i=1}^n\left(\sum_{|\alpha|=m}\alpha_i^2|b_\alpha|^2\right)^{1/2}.
\]
{Combining \eqref{eq:P_t-q0-lower} and \eqref{eq:S-lower} gives, for a suitable $\tau\in[0,1]$,}
\[
 |P|_q\le |P_t|_{q_m}^{\,1-\tau}
 \left(\frac{S(P_t)}{\sqrt m}\right)^\tau.
\]
{The Bohnenblust--Hille inequality controls $|P_t|_{q_m}$, the derivative estimate controls $S(P_t)$, and the rescaling gives $\|P_t\|_\infty\le\|P\|_p$. Hence the right-hand side is bounded in terms of $D_m$ and $\|P\|_p$, without polarization. The precise choice of $\tau$ and the complete calculation are given in Lemma~\ref{lem:coefficient-transfer}.}}

\medskip
\noindent\textbf{Scalar $H^1$ coefficient contraction and tensorization.}\par\smallskip
Near $p=m$ the useful coefficient exponent $q=p/(p-m)$ is large, and {the coefficient-dependent rescaling mechanism of Lemma~\ref{lem:coefficient-transfer}} is no longer the relevant mechanism.  {The argument begins with a one-variable $H^1$ estimate.}  Take, for example,
\[
 f(z)=a_0+a_1z+a_2z^2.
\]
{For $1\le r<\infty$, $H^r$ denotes the usual analytic Hardy space on $\mathbb T$, with
\[
 \|f\|_{H^r}:=\left(\int_{\mathbb T}|f(\zeta)|^r\,d\mathfrak m_1(\zeta)\right)^{1/r}.
\]
In particular, $\|f\|_{H^1}=\int_{\mathbb T}|f(\zeta)|\,d\mathfrak m_1(\zeta)$.} We use the classical Riesz factorization:
every $f\in H^1$ can be written as $f=gh$ with $g,h\in H^2$ and
$\|f\|_{H^1}=\|g\|_{H^2}\|h\|_{H^2}$; see, for instance,
\cite[Section~2.4]{Duren1970}.  Coefficient multipliers from $H^1$ into
$\ell_q$ form a classical subject; a convenient general reference is
\cite[Corollary~8.3]{BlascoPavlovic2011}.  {The multiplier problem itself is classical. {What is needed in Lemma~\ref{lem:scalar-radius} is more specific:}
{for $q\ge16$, the explicit norm-one geometric multiplier $(\rho_q^k)_{k\ge0}$, where
$\rho_q:=\exp\!\left(-\frac{\log q}{qM_q}\right)$ and $M_q$ is defined in
\eqref{eq:scalar-radius-parameters}, with $\rho_q$ quantitatively tracked as $q\to\infty$.}}

{For the quadratic polynomial above, define}
\[
 f_{\rho_q}(z):=a_0+\rho_q a_1z+\rho_q^2a_2z^2.
\]
{Lemma~\ref{lem:scalar-radius} shows that this parameter satisfies}
\[
 \bigl(|a_0|^q+\rho_q^q|a_1|^q+\rho_q^{2q}|a_2|^q\bigr)^{1/q}
 \leq \int_{\mathbb T}|f(\zeta)|\,d\mathfrak m_1(\zeta)
 =\|f\|_{H^1}.
\]
{More generally, in this range,} Lemma~\ref{lem:scalar-radius} gives, for $f(z)=\sum_{k=0}^N a_kz^k$,
\[
 \left(\sum_{k=0}^N|a_k|^q\rho_q^{qk}\right)^{1/q}
 \leq\left\|\sum_{k=0}^N a_kz^k\right\|_{H^1}.
\]
The important point is that the multiplicative constant in this inequality is
exactly $1$: there is no additional factor $C>1$ multiplying the $H^1$ norm.  This
is what permits the estimate to be applied successively in different variables
without producing a factor that grows with the dimension.

{Here tensorization means applying the same one-variable coefficient contraction successively in each variable. Since each application has constant $1$, no additional multiplicative constant accumulates with the number of variables.} The mechanism can already be seen in two variables. Consider
\[
 F:\mathbb C^2\longrightarrow\mathbb C,\qquad
 F(z,w)=a+bz+cw+dzw.
\]
Lemma~\ref{lem:scalar-radius} will justify, with constant $1$, the operation
\[
 \sum_{k=0}^N c_k w^k
 \longmapsto
 \sum_{k=0}^N \rho_q^k c_k w^k.
\]
Applied first in the $w$-variable, with $z$ fixed, this will transform
\[
 F(z,w)=(a+bz)+(c+dz)w
\]
into
\[
 (a+bz)+\rho_q(c+dz)w
 =a+bz+\rho_q cw+\rho_q dzw.
\]
Applying the same operation in the $z$-variable will then give
\[
 a+\rho_q bz+\rho_q cw+\rho_q^2dzw.
\]
Thus each variable occurring in a monomial contributes one factor $\rho_q$ in
this example: $bz$ and $cw$ acquire one factor, while $dzw$ acquires two.

{The same argument iterates in $n$ variables.}  For a monomial
\[
 z^\alpha=z_1^{\alpha_1}\cdots z_n^{\alpha_n},
\]
the contraction in the first variable contributes $\rho_q^{\alpha_1}$, the second
contributes $\rho_q^{\alpha_2}$, and so on.  Hence the total weight is
\[
 \rho_q^{\alpha_1}\cdots\rho_q^{\alpha_n}
 =\rho_q^{|\alpha|}.
\]
In particular, if the polynomial is $d$-homogeneous, every coefficient receives the
same weight $\rho_q^d$, and the $n$ successive one-variable estimates still have
multiplicative constant $1$.

When the tensorized estimate is applied to the derivatives of the rescaled
$m$-homogeneous polynomial, every derivative has degree $m-1$ and therefore pays
$\rho_q^{-(m-1)}$ when its coefficient norm is recovered.  Combined with the
lower-range coefficient reduction and the derivative estimate, this produces the
quantitative bound from which both parts of Theorem~\ref{thm:B} follow.
{Lemma~\ref{lem:scalar-radius} establishes the norm-one radius $\rho_q$ and its coefficient contraction, while Lemma~\ref{lem:tensor-radius} gives the corresponding dimension-free tensorization. Since $q=p/(p-m)\to\infty$ as $p\downarrow m$, the asymptotic rate of $\rho_q\to1$ determines the factor $\rho_q^{-(m-1)}$ and hence the two strip widths in Theorem~\ref{thm:B}.}

\medskip
\noindent\textbf{Multiplicity patterns.}\par\smallskip
Parts~\textup{(ii)}--\textup{(iii)} of Theorem~\ref{thm:C} quantify the stability of the endpoint
identity $D_{m,m}=m$ on two asymptotic scales. For that purpose it matters how
many times each variable occurs in a monomial.

For example, the three degree-six monomials
\[
 z_1^3z_4^2z_7,
 \qquad
 z_2^3z_5^2z_8,
 \qquad
 z_3^3z_6^2z_9
\]
all have the same multiplicity pattern
\[
 (3,2,1).
\]
The names of the variables are different, but combinatorially each monomial uses
one variable three times, one twice, and one once.  By contrast,
\[
 z_1^4z_2z_3
\]
{has pattern $(4,1,1)$, although the total degree is still six.  Thus a multiplicity pattern forgets the coordinate labels and retains the exponent structure of the monomial.}

The guiding idea is to isolate the monomials with one prescribed
multiplicity pattern and rewrite that part of the polynomial as a multilinear
form.  For the pattern $(3,2,1)$, the three roles are: a variable carrying
exponent $3$, a second variable carrying exponent $2$, and a third carrying
exponent $1$.  {Separate coordinate blocks allow these three roles to be chosen
independently, after which the surviving coefficients form a trilinear form.}

Consider the concrete case $n=9$ and
\[
 S_1:=\{1,2,3\},\qquad S_2:=\{4,5,6\},\qquad S_3:=\{7,8,9\}.
\]
We choose the variable with exponent $3$ from $S_1$, the variable with exponent
$2$ from $S_2$, and the variable with exponent $1$ from $S_3$. Thus
\[
 i\in S_1,\quad j\in S_2,\quad k\in S_3
 \quad\longmapsto\quad z_i^3z_j^2z_k.
\]
For example,
\[
 (1,4,7)\longmapsto z_1^3z_4^2z_7.
\]
If $e_r$ denotes the $r$-th canonical vector, this same monomial corresponds to
\[
 3e_1+2e_4+e_7,
 \qquad z^{3e_1+2e_4+e_7}=z_1^3z_4^2z_7.
\]
This example leads to the general definition. For pairwise disjoint subsets
$S_1,S_2,S_3\subseteq\{1,\ldots,n\}$, put $S:=(S_1,S_2,S_3)$ and define
\[
 \Phi_{(3,2,1)}:S_1\times S_2\times S_3
 \longrightarrow\{\alpha\in\mathbb N_0^n:|\alpha|=6\},\qquad
 \Phi_{(3,2,1)}(i,j,k):=3e_i+2e_j+e_k.
\]
If
\[
 P:\mathbb C^n\longrightarrow\mathbb C,\qquad
 P(z)=\sum_{|\alpha|=6}a_\alpha z^\alpha,
\]
then the block projection keeps exactly the monomials produced by this map:
\[
 (\mathcal Q_{(3,2,1),S}P)(z)
 :=\sum_{i\in S_1}\sum_{j\in S_2}\sum_{k\in S_3}
 a_{3e_i+2e_j+e_k}z_i^3z_j^2z_k.
\]
All other monomials are removed. {For a finite set $E$ and
$1\le r\le\infty$, we write $\ell_r(E):=\mathbb C^E$ with its usual
$\ell_r$ norm.} For $6<p\le\infty$, the surviving coefficients define
\[
 B_{(3,2,1),S}:\ell_{p/3}(S_1)\times\ell_{p/2}(S_2)\times\ell_p(S_3)
 \longrightarrow\mathbb C,
\]
\[
 B_{(3,2,1),S}(x,y,w)
 :=\sum_{i\in S_1}\sum_{j\in S_2}\sum_{k\in S_3}
 a_{3e_i+2e_j+e_k}x_i y_j w_k.
\]
The two descriptions are related by
\[
 (\mathcal Q_{(3,2,1),S}P)(z)
 =B_{(3,2,1),S}\bigl((z_i^3)_{i\in S_1},(z_j^2)_{j\in S_2},(z_k)_{k\in S_3}\bigr).
\]
{Thus the block projection retains the $(3,2,1)$ coefficients and identifies them with an anisotropic trilinear form to which the mixed estimates apply.}

The same association works for every partition
$\lambda=(\lambda_1,\ldots,\lambda_s)$ of $m$. For $m<p\le\infty$ and disjoint
blocks $S_1,\ldots,S_s$, put $S:=(S_1,\ldots,S_s)$ and define
\[
 \Phi_\lambda:S_1\times\cdots\times S_s
 \longrightarrow\{\alpha\in\mathbb N_0^n:|\alpha|=m\},
 \qquad
 \Phi_\lambda(j_1,\ldots,j_s)
 :=\lambda_1e_{j_1}+\cdots+\lambda_se_{j_s}.
\]
The corresponding block projection retains the coefficients indexed by the image of
$\Phi_\lambda$, and these coefficients define an $s$-linear form
\[
 B_{\lambda,S}:
 \ell_{p/\lambda_1}(S_1)\times\cdots\times
 \ell_{p/\lambda_s}(S_s)\longrightarrow\C.
\]
Section~\ref{sec:mixed} gives the formal construction and the corresponding projection estimate.
Mixed Littlewood estimates and their anisotropic transfer control this one pattern.
Averaging over block decompositions recovers the original coefficients with a
computable survival fraction, and summing over the possible partitions gives the
global estimate used in Theorem~\ref{thm:C}.  The mixed Littlewood and anisotropic estimates are recalled with their references at the point of use.

\clearpage
\noindent{\Large\bfseries Proof architecture}\par\medskip

\begin{center}
\vspace{0.3\baselineskip}
\resizebox{0.98\textwidth}{!}{%
\begin{tikzpicture}[
  x=1cm,y=0.82cm,
  result/.style={draw=black!55, line width=0.45pt,
    rounded corners=1.5pt, fill=white, text=black,
    align=center, text width=3.12cm, minimum height=0.91cm,
    inner xsep=5pt, inner ysep=5pt, font=\fontsize{9}{11}\selectfont},
  main/.style={result, draw=black!85, line width=0.8pt,
    fill=black!3},
  arr/.style={-{Stealth[length=1.65mm,width=1.2mm]},
    draw=black!70, line width=0.5pt},
  lane/.style={anchor=west, text=black,
    font=\small\scshape},
  rule/.style={draw=black!20, line width=0.35pt}
]
\node[lane] at (-1.7,0) {Upper range};
\draw[rule] (1.35,0) -- (14,0);
\node[result] (coefficient) at (0,-1.05)
  {Coefficient transfer\\[-1pt]\footnotesize Lemma~\ref*{lem:coefficient-transfer}};
\node[result] (derivative) at (0,-2.6)
  {Derivative estimate\\[-1pt]\footnotesize Lemma~\ref*{lem:derivative-entropy}};
\node[result] (transfer) at (6.15,-1.825)
  {Upper-range transfer\\[-1pt]\footnotesize Proof of Theorem A};
\node[main] (theoremA) at (12.3,-1.05)
  {\textbf{Theorem A}\\[-1pt]\footnotesize \auditfix{Asymptotic transfer}};
\draw[arr] (coefficient.east) -- (transfer.north west);
\draw[arr] (derivative.east) -- (transfer.south west);
\draw[arr] (transfer.north east) -- (theoremA.west);

\node[lane] at (-1.7,-4.0) {Lower range};
\draw[rule] (1.4,-4.0) -- (14,-4.0);
\node[result] (scalar) at (0,-5.05)
  {Scalar radius\\[-1pt]\footnotesize Lemma~\ref*{lem:scalar-radius}};
\node[result] (tensor) at (4.1,-5.05)
  {Tensorization\\[-1pt]\footnotesize Lemma~\ref*{lem:tensor-radius}};
\node[result] (lowerinput) at (4.1,-6.6)
  {Coefficient and\\derivative estimates\\[-1pt]\footnotesize Lemmas~\ref*{lem:lower-coefficient-reduction} and~\ref*{lem:derivative-entropy}};
\node[result] (lower) at (8.2,-5.825)
  {Lower-range bound\\[-1pt]\footnotesize Proposition~\ref*{prop:technical-lower}};
\node[main] (theoremBnear) at (12.3,-5.825)
  {\textbf{Theorem B}\\[-1pt]\footnotesize {Near-endpoint\\estimates}};
\draw[arr] (scalar.east) -- (tensor.west);
\draw[arr] (tensor.east) -- (lower.north west);
\draw[arr] (lowerinput.east) -- (lower.south west);
\draw[arr] (lower.east) -- (theoremBnear.west);

\node[lane] at (-1.7,-8.2) {Multiplicity patterns and endpoint stability};
\draw[rule] (7.65,-8.2) -- (14,-8.2);
\node[result] (sources) at (0,-9.35)
  {Mixed estimate\\and transfer\\[-1pt]\footnotesize Lemmas~\ref*{lem:mixed-source} and~\ref*{lem:specialized-transfer}};
\node[result] (anisotropic) at (4.1,-9.35)
  {Anisotropic estimate\\[-1pt]\footnotesize Lemma~\ref*{lem:anisotropic}};
\node[result] (pattern) at (8.2,-9.35)
  {One-pattern estimate\\[-1pt]\footnotesize Lemma~\ref*{lem:one-pattern}};
\node[result] (sum) at (12.3,-9.35)
  {Sum over patterns\\[-1pt]\footnotesize Lemma~\ref*{lem:global-sum}};
\node[result] (projection) at (4.1,-10.9)
  {Block projection\\[-1pt]\footnotesize Lemma~\ref*{lem:block-projection}};
\node[result] (global) at (12.3,-10.9)
  {Global pattern bound\\[-1pt]\footnotesize Proposition~\ref*{prop:pattern-global}};
\node[result] (scalarreuse) at (4.1,-12.55)
  {{Theorem B(i)}\\[-1pt]\footnotesize {Lower-endpoint input}};
\node[result] (endpoint) at (8.2,-14.05)
  {Coefficient extraction\\[-1pt]\footnotesize
    \minorfix{Proof of\\[-1pt]\mbox{Theorem C(i)}}};
\node[main] (theoremC) at (12.3,-12.55)
  {\textbf{Theorem C}\\[-1pt]\footnotesize Exact endpoint and stability};
\draw[arr] (sources.east) -- (anisotropic.west);
\draw[arr] (anisotropic.east) -- (pattern.west);
\draw[arr] (projection.east) -- (pattern.south west);
\draw[arr] (pattern.east) -- (sum.west);
\draw[arr] (sum.south) -- (global.north);
\draw[arr] (global.south) -- (theoremC.north);
\draw[arr] (scalarreuse.east) -- (theoremC.west);
\draw[arr] (endpoint.east) -- (theoremC.south west);

\end{tikzpicture}%
}

\medskip
\footnotesize
\minorfix{Arrows indicate the main dependencies. Repeated references record reuse of
the same estimate. {The upper-range transfer also yields the
support-preserving consequence in Remark~\ref{rem:support-transfer}.}}
\end{center}
\part{Upper range and transfer}

\section{Upper Hardy--Littlewood range}\label{sec:upper-HL}

{In this section we prove the upper-range transfer estimate leading to Theorem~\ref{thm:A}. We first establish the derivative estimate needed to control the rescaled polynomial, then introduce the coefficient-dependent rescaling and combine the two estimates to obtain the transfer from the Bohnenblust--Hille endpoint.}

\subsection{Derivative estimates}

{Lemma~\ref{lem:derivative-entropy} is the derivative estimate used in the upper-range argument.}

\begin{lemma}\label{lem:derivative-entropy}
Let
\[
 Q:\mathbb C^n\longrightarrow\mathbb C,\qquad
 Q(z)=\sum_{|\alpha|=m}b_\alpha z^\alpha
\]
be an $m$-homogeneous polynomial, and set
\[
 {S(Q):=\sum_{i=1}^n|\partial_iQ|_2}
 =\sum_{i=1}^n
 \left(\sum_{|\alpha|=m}\alpha_i^2|b_\alpha|^2\right)^{1/2}.
\]
Then
\begin{equation}\label{eq:derivative-l1-pointwise}
 \sum_{i=1}^n|\partial_iQ(z)|
 \leq \frac\pi2 m\|Q\|_\infty,
 \qquad z\in\mathbb T^n,
\end{equation}
and
\begin{equation}\label{eq:S-upper}
 S(Q)\leq \frac\pi2 m2^{(m-1)/2}\|Q\|_\infty.
\end{equation}
\end{lemma}

\begin{proof}
Fix $z\in\mathbb T^n$ and put
\[
 c_i:=z_i\partial_iQ(z)\qquad(1\leq i\leq n).
\]
For $\varepsilon=(\varepsilon_1,\ldots,\varepsilon_n)\in\{-1,1\}^n$, define
{$g_\varepsilon:\mathbb R\to\mathbb C$ by
\[
\begin{aligned}
 g_\varepsilon(s)&:=Q(z_1e^{i\varepsilon_1s},\ldots,z_ne^{i\varepsilon_ns})\\
 &=\sum_{|\alpha|=m}b_\alpha z^\alpha
 e^{\,i(\alpha_1\varepsilon_1+\cdots+\alpha_n\varepsilon_n)s}.
\end{aligned}
\]}

For every $|\alpha|=m$,
\[
 \left|\alpha_1\varepsilon_1+\cdots+\alpha_n\varepsilon_n\right|
 \leq\alpha_1+\cdots+\alpha_n=m.
\]
Hence $g_\varepsilon$ is a trigonometric polynomial of the form
\[
 g_\varepsilon(s)=\sum_{k=-m}^{m}d_ke^{iks}.
\]
The classical Bernstein inequality for trigonometric polynomials states that
\[
 h(s)=\sum_{k=-m}^{m}d_ke^{iks}
 \quad\Longrightarrow\quad
 \|h'\|_\infty\leq m\|h\|_\infty;
\]
see, for example, Dai and Xu \cite[Theorem~4.1.3]{DaiXu2013}.
Since
\[
 \|g_\varepsilon\|_\infty\leq\|Q\|_\infty
\]
and
\[
 g_\varepsilon'(0)
 =
 i\sum_{i=1}^n\varepsilon_i z_i\partial_iQ(z)
 =
 i\sum_{i=1}^n\varepsilon_i c_i,
\]
we obtain
{
\begin{equation}\label{eq:bernstein-sign-bound}
 \left|\sum_{i=1}^n\varepsilon_i c_i\right|
 \leq m\|Q\|_\infty.
\end{equation}
}

For $\theta\in[0,2\pi]$, {choose $\varepsilon\in\{-1,1\}^n$ by
\[
 \varepsilon_i:=
 \begin{cases}
  1,&\operatorname{Re}(e^{i\theta}c_i)\geq0,\\
 -1,&\operatorname{Re}(e^{i\theta}c_i)<0,
 \end{cases}
 \qquad 1\leq i\leq n.
\]
}Here $\operatorname{Re}w$ denotes the real part of $w$. {With this choice of signs, \eqref{eq:bernstein-sign-bound} gives}
{
\[
 \sum_{i=1}^n\left|\operatorname{Re}(e^{i\theta}c_i)\right|
 =\operatorname{Re}\!\left(e^{i\theta}\sum_{i=1}^n\varepsilon_i c_i\right)
 \leq\left|\sum_{i=1}^n\varepsilon_i c_i\right|\leq m\|Q\|_\infty.
\]
}
Integrating with respect to normalized Lebesgue measure on $[0,2\pi]$ gives
{
\begin{equation}\label{eq:integrated-real-parts}
 \sum_{i=1}^n\frac1{2\pi}\int_0^{2\pi}
 \left|\operatorname{Re}(e^{i\theta}c_i)\right|\,d\theta
 \leq m\|Q\|_\infty.
\end{equation}
}
If $c_i=|c_i|e^{i\varphi_i}$, then
{
\begin{equation}\label{eq:real-part-average}
 \frac1{2\pi}\int_0^{2\pi}\left|\operatorname{Re}(e^{i\theta}c_i)\right|\,d\theta
 =|c_i|\frac1{2\pi}\int_0^{2\pi}|\cos\theta|\,d\theta=\frac2\pi|c_i|.
\end{equation}
Combining \eqref{eq:integrated-real-parts} and \eqref{eq:real-part-average} yields
\[
 \sum_{i=1}^n|c_i|
 \leq\frac\pi2m\|Q\|_\infty.
\]}
Since $c_i=z_i\partial_iQ(z)$ and $|z_i|=1$,
\[
 |c_i|=|z_i\partial_iQ(z)|=|\partial_iQ(z)|,
 \qquad 1\leq i\leq n.
\]
{Substituting this identity into the preceding estimate gives
\[
 \sum_{i=1}^n|\partial_iQ(z)|=\sum_{i=1}^n|c_i|
 \leq\frac{\pi}{2}m\|Q\|_\infty,
\]}
which is \eqref{eq:derivative-l1-pointwise}.

{To prove \eqref{eq:S-upper},} the polynomial Khintchine inequality
\begin{equation}\label{eq:polynomial-khintchine}
 \|f\|_{L^2(\T^n)}\leq2^{d/2}\|f\|_{L^1(\T^n)}
\end{equation}
holds for every complex $d$-homogeneous analytic polynomial $f$; see
\cite[Lemma~5.1]{BayartPellegrinoSeoane2014}, ultimately based on
\cite[Theorem~9]{Bayart2002}. Since $\partial_iQ$ is $(m-1)$-homogeneous, {we write explicitly}
{
\[
 \partial_iQ(z)
 =\sum_{\substack{|\alpha|=m\\ \alpha_i\geq1}}
 \alpha_i b_\alpha z^{\alpha-e_i}.
\]
}
{For fixed $i$, the map $\alpha\mapsto\alpha-e_i$ is injective on
$\{\alpha\in\mathbb N_0^n:|\alpha|=m,\ \alpha_i\ge1\}$. Hence the monomials
$z^{\alpha-e_i}$ occurring in $\partial_iQ$ are distinct characters of $\T^n$
and are pairwise orthogonal in $L^2(\T^n)$. Thus Parseval's identity can be
written explicitly as
\begin{align*}
 \|\partial_iQ\|_{L^2(\T^n)}^2
 &=\int_{\T^n}\left|
   \sum_{\substack{|\alpha|=m\\ \alpha_i\ge1}}
   \alpha_i b_\alpha z^{\alpha-e_i}\right|^2
   \,d\mathfrak m_n(z)\\
 &=\sum_{\substack{|\alpha|=m\\ \alpha_i\ge1}}
   \sum_{\substack{|\beta|=m\\ \beta_i\ge1}}
   \alpha_i\beta_i b_\alpha\overline{b_\beta}
   \int_{\T^n}z^{\alpha-e_i}\overline{z^{\beta-e_i}}
   \,d\mathfrak m_n(z)\\
 &=\sum_{|\alpha|=m}\alpha_i^2|b_\alpha|^2,
\end{align*}
where the last equality follows from the orthogonality of the characters.}
{Combining this identity with \eqref{eq:polynomial-khintchine} and then
\eqref{eq:derivative-l1-pointwise}, we obtain
\[
 S(Q)=\sum_{i=1}^n\|\partial_iQ\|_{L^2(\T^n)}
 \overset{\eqref{eq:polynomial-khintchine}}{\leq}
 2^{(m-1)/2}\sum_{i=1}^n\|\partial_iQ\|_{L^1(\T^n)}
 \overset{\eqref{eq:derivative-l1-pointwise}}{\leq}
 \frac\pi2m2^{(m-1)/2}\|Q\|_\infty.
\]
}
\end{proof}

\subsection{Coefficient-dependent rescaling}

{The next lemma is the main step in the proof of Theorem~\ref{thm:A}. Its crucial conclusion is the coefficient transfer estimate \eqref{eq:numerical-transfer}; together with the derivative estimate above and the norm bound \eqref{eq:P_t-norm-transfer}, it leads directly to the proof of Theorem~\ref{thm:A} in the next subsection.}

\begin{lemma}\label{lem:coefficient-transfer}
Let $m\geq2$ and $2m\leq p<\infty$. Put
\begin{equation}\label{eq:upper-range-q-tau}
 \frac1q:=\frac1{q_m}-\frac1p,
 \qquad
 \tau:=\frac{2m}{p}.
\end{equation}
Let
\[
 P:\mathbb C^n\longrightarrow\mathbb C,\qquad
 P(z)=\sum_{|\alpha|=m}a_\alpha z^\alpha\neq0.
\]
For $|\alpha|=m$ and $1\leq i\leq n$, define
\[
 \mu_\alpha:=\frac{|a_\alpha|^q}{|P|_q^q},
 \qquad
 t_i:=\frac1m\sum_{|\alpha|=m}\alpha_i\mu_\alpha.
\]
Set
\[
 P_t:\mathbb C^n\longrightarrow\mathbb C,\qquad
 P_t(z_1,\ldots,z_n)
 :=
 P(t_1^{1/p}z_1,\ldots,t_n^{1/p}z_n).
\]
{Here $S$ is the same functional introduced in Lemma~\ref{lem:derivative-entropy}; in particular,
\[
 S(P_t)=\sum_{i=1}^n|\partial_iP_t|_2.
\]}
Then
\begin{equation}\label{eq:P_t-norm-transfer}
 \|P_t\|_\infty\leq\|P\|_p,
\end{equation}
and
\begin{equation}\label{eq:numerical-transfer}
 |P|_q\leq |P_t|_{q_m}^{\,1-\tau}
       \left(\frac{S(P_t)}{\sqrt m}\right)^\tau.
\end{equation}
\end{lemma}

\begin{proof}
Coefficients equal to zero may be omitted. {By definition $t_i\geq0$ for every $i$, and since
$\sum_{|\alpha|=m}\mu_\alpha=1$ and $|\alpha|=m$,
\[
 \sum_{i=1}^nt_i
 =\frac1m\sum_{|\alpha|=m}\mu_\alpha\sum_{i=1}^n\alpha_i=1.
\]
If $t_i=0$, then
\[
 0=mt_i=\sum_{|\alpha|=m}\alpha_i\mu_\alpha,
\]
and the nonnegativity of the summands implies $\alpha_i=0$ for every
$\alpha$ with $\mu_\alpha>0$. Hence coordinates with $t_i=0$ may also be
omitted, and {all logarithms in the two Jensen estimates leading to \eqref{eq:P_t-q0-lower} and \eqref{eq:S-lower} have strictly positive arguments}.}
{For $z\in\mathbb C^n$ with $|z_i|\leq1$ for $1\leq i\leq n$,
\[
 \sum_{i=1}^n|t_i^{1/p}z_i|^p
 =\sum_{i=1}^nt_i|z_i|^p\leq\sum_{i=1}^nt_i=1.
\]
Hence $(t_1^{1/p}z_1,\ldots,t_n^{1/p}z_n)$ belongs to the unit ball of
$\ell_p^n$, and therefore
\[
 |P_t(z)|
 =\bigl|P(t_1^{1/p}z_1,\ldots,t_n^{1/p}z_n)\bigr|
 \leq \|P\|_p.
\]
Taking the supremum over all $z\in\mathbb C^n$ with $|z_i|\leq1$
for $1\leq i\leq n$ gives explicitly
\[
 \|P_t\|_\infty
 =\sup_{|z_1|,\ldots,|z_n|\leq1}|P_t(z)|
 \leq\|P\|_p,
\]
which is \eqref{eq:P_t-norm-transfer}.}

{It remains to prove \eqref{eq:numerical-transfer}.}

{We use the Shannon entropy $H$ from \eqref{eq:Shannon-entropy}.} We also use Jensen's inequality for the
concave function $\log$ in the form
\begin{equation}\label{eq:log-sum-jensen}
 \log\left(\sum_{j=1}^N u_j\right)
 \geq
 \sum_{j=1}^N v_j\log\frac{u_j}{v_j},
 \qquad
 u_j\geq0,\quad v_j\geq0,\quad
 \sum_{j=1}^Nv_j=1,
\end{equation}
with zero-weight terms omitted \minorfix{and $u_j>0$ whenever $v_j>0$};
see Hardy, Littlewood, and P\'olya
\cite[Theorem~2, p.~17]{HardyLittlewoodPolya1952}.

\medskip
{We first establish the rescaled coefficient-norm estimate
\eqref{eq:P_t-q0-lower}.}\par\smallskip
Write
\begin{equation}\label{eq:I-definition}
 \left[
 \begin{aligned}
 I&:=mH(t)-H(\mu)
 =-m\sum_{i=1}^n t_i\log t_i
   +\sum_{|\alpha|=m}\mu_\alpha\log\mu_\alpha,\\
 b_\alpha&:=a_\alpha t^{\alpha/p},\\
 t^{\alpha/p}&:=\prod_{\substack{1\leq i\leq n\\ \alpha_i>0}}
 t_i^{\alpha_i/p}.
 \end{aligned}
 \right.
\end{equation}

The relation between the two coefficient exponents in
\eqref{eq:upper-range-q-tau} is
\begin{equation}\label{eq:q0q-identity}
 \frac{q_m}{q}=1-\frac{q_m}{p}.
\end{equation}
{Since $b_\alpha$ is the coefficient of $z^\alpha$ in $P_t$,
\eqref{eq:q0q-identity} allows the normalized $q_m$-coefficient norm of
$P_t$ to be written directly in terms of the weights $\mu_\alpha$ and the
scales $t_i$.  For every $|\alpha|=m$,}
\begin{align*}
 \mu_\alpha^{q_m/q}
 &=\mu_\alpha^{1-q_m/p}
 =\mu_\alpha\exp\!\left(-\frac{q_m}{p}\log\mu_\alpha\right),\\
 t^{q_m\alpha/p}
 &=\prod_{\substack{1\leq i\leq n\\ \alpha_i>0}}
 t_i^{q_m\alpha_i/p}
 =\exp\!\left(
 \frac{q_m}{p}
 \sum_{\substack{1\leq i\leq n\\ \alpha_i>0}}
 \alpha_i\log t_i\right).
\end{align*}
Since $b_\alpha=a_\alpha t^{\alpha/p}$ and
$\mu_\alpha=|a_\alpha|^q/|P|_q^q$, we therefore have
\[
 \left(\frac{|P_t|_{q_m}}{|P|_q}\right)^{q_m}
 =\sum_{|\alpha|=m}\mu_\alpha^{q_m/q}t^{q_m\alpha/p}
 =\sum_{|\alpha|=m}\mu_\alpha
 \exp\!\left(\frac{q_m}{p}\left(
 \sum_{\substack{1\leq i\leq n\\ \alpha_i>0}}\alpha_i\log t_i-\log\mu_\alpha
 \right)\right).
\]
{Since $\sum_{|\alpha|=m}\mu_\alpha=1$, Jensen's inequality
\eqref{eq:log-sum-jensen} applied to the preceding identity gives
\begin{align*}
 q_m\log\frac{|P_t|_{q_m}}{|P|_q}
 &=\log\!\left[
 \sum_{|\alpha|=m}\mu_\alpha
 \exp\!\left(
 \frac{q_m}{p}
 \left(
 \sum_{\substack{1\leq i\leq n\\ \alpha_i>0}}
 \alpha_i\log t_i-\log\mu_\alpha
 \right)\right)\right]\\
 &\geq
 \frac{q_m}{p}
 \sum_{|\alpha|=m}\mu_\alpha
 \left(
 \sum_{i=1}^n\alpha_i\log t_i-\log\mu_\alpha
 \right).
\end{align*}}
Dividing by $q_m$ gives
\[
 \log\frac{|P_t|_{q_m}}{|P|_q}
 \geq\frac1p\left(
 \sum_{|\alpha|=m}\sum_{i=1}^n\mu_\alpha\alpha_i\log t_i
 -\sum_{|\alpha|=m}\mu_\alpha\log\mu_\alpha\right)
 =\frac1p\left(m\sum_{i=1}^n t_i\log t_i+H(\mu)\right)=-\frac Ip.
\]
Exponentiating the preceding inequality gives
\begin{equation}\label{eq:P_t-q0-lower}
 |P_t|_{q_m}\geq |P|_q \e^{-I/p}.
\end{equation}

\medskip
{We next establish the derivative-sum estimate \eqref{eq:S-lower}.}\par\smallskip
{Define normalized weights on pairs $(i,\alpha)$ by}
\begin{equation}\label{eq:joint-nu}
 \nu(i,\alpha):=\frac{\alpha_i\mu_\alpha}{m}.
\end{equation}
{Summing these weights over $\alpha$ and over $i$, respectively, gives}
{
\begin{equation}\label{eq:nu-marginals}
 \sum_{|\alpha|=m}\nu(i,\alpha)=t_i,
 \qquad
 \sum_{i=1}^n\nu(i,\alpha)=\mu_\alpha.
\end{equation}
}
{For each $i$ with $t_i>0$, normalize the $\alpha$-weights by}
\begin{equation}\label{eq:conditional-beta}
 \beta_{\alpha\mid i}:=\frac{\nu(i,\alpha)}{t_i};
\end{equation}
then $\sum_{|\alpha|=m}\beta_{\alpha\mid i}=1$. {For each $i$, put $s_i:=|\partial_iP_t|_2$. By the coefficient formula for $P_t$ and the definition of $S$ in Lemma~\ref{lem:derivative-entropy},
\begin{equation}\label{eq:si-definition}
 s_i
 =\left(\sum_{|\alpha|=m}
       \alpha_i^2|a_\alpha|^2t^{2\alpha/p}\right)^{1/2},
 \qquad\text{and hence}\qquad S(P_t)=\sum_{i=1}^n s_i.
\end{equation}}

{If $t_i>0$, then some $\alpha$ satisfies $\alpha_i\mu_\alpha>0$; for this $\alpha$, every $t_j$ with $\alpha_j>0$ is positive, so the corresponding term in $s_i$ is nonzero. Hence $s_i>0$, and {the logarithms in \eqref{eq:S-Jensen-outer} and \eqref{eq:S-Jensen-inner} are well defined}.}
{
Since $\sum_{i=1}^n t_i=1$, the numbers $t_i$ are probability weights. Applying
Jensen's inequality \eqref{eq:log-sum-jensen} directly gives
\begin{equation}\label{eq:S-Jensen-outer}
 \log\frac{S(P_t)}{|P|_q}
 \overset{\eqref{eq:log-sum-jensen}}{\geq}
 \sum_{i=1}^n t_i\log\frac{s_i}{t_i|P|_q}
 =H(t)+\sum_{i=1}^n t_i\log\frac{s_i}{|P|_q}.
\end{equation}
For the second application, fix $i$ with $t_i>0$. By
\eqref{eq:conditional-beta}, $\sum_{|\alpha|=m}\beta_{\alpha\mid i}=1$, and
\[
 \frac{s_i^2}{|P|_q^2}
 =\sum_{|\alpha|=m}r_{i,\alpha},
 \qquad
 r_{i,\alpha}:=\alpha_i^2\mu_\alpha^{2/q}t^{2\alpha/p}.
\]
Hence another direct application of \eqref{eq:log-sum-jensen} gives
\begin{equation}\label{eq:S-Jensen-inner}
 2\log\frac{s_i}{|P|_q}
 \overset{\eqref{eq:log-sum-jensen}}{\geq}
 \sum_{|\alpha|=m}\beta_{\alpha\mid i}
 \log\frac{r_{i,\alpha}}{\beta_{\alpha\mid i}}.
\end{equation}
}
{Multiplying \eqref{eq:S-Jensen-inner} by $t_i/2$ gives}
\[
 t_i\log\frac{s_i}{|P|_q}
 \geq
 \frac{t_i}{2}
 \sum_{|\alpha|=m}\beta_{\alpha\mid i}
 \log\frac{r_{i,\alpha}}{\beta_{\alpha\mid i}}.
\]
Summing over $1\leq i\leq n$,
\[
 \sum_{i=1}^n t_i\log\frac{s_i}{|P|_q}
 \geq
 \frac12\sum_{i=1}^n\sum_{|\alpha|=m}
 t_i\beta_{\alpha\mid i}
 \log\frac{r_{i,\alpha}}{\beta_{\alpha\mid i}}.
\]
{By \eqref{eq:conditional-beta},
\[
 t_i\beta_{\alpha\mid i}
 =t_i\frac{\nu(i,\alpha)}{t_i}
 =\nu(i,\alpha),
\]
and hence the preceding estimate becomes
\[
 \sum_{i=1}^n t_i\log\frac{s_i}{|P|_q}
 \geq
 \frac12\sum_{i=1}^n\sum_{|\alpha|=m}
 \nu(i,\alpha)
 \log\frac{r_{i,\alpha}}{\beta_{\alpha\mid i}}.
\]
Substituting this bound for
$\sum_{i=1}^n t_i\log(s_i/|P|_q)$ in \eqref{eq:S-Jensen-outer} gives
\begin{equation}\label{eq:S-Jensen}
 \log\frac{S(P_t)}{|P|_q}
 \geq H(t)+\frac12
 \sum_{i=1}^n\sum_{|\alpha|=m}\nu(i,\alpha)
 \log\frac{r_{i,\alpha}}{\beta_{\alpha\mid i}}.
\end{equation}
Moreover, \eqref{eq:conditional-beta} and \eqref{eq:joint-nu} give directly
\[
 \beta_{\alpha\mid i}
 =\frac{\nu(i,\alpha)}{t_i}
 =\frac{\alpha_i\mu_\alpha}{m t_i}.
\]
}
{For pairs $(i,\alpha)$ with $\nu(i,\alpha)>0$ (equivalently,
$\beta_{\alpha\mid i}>0$), one has $\alpha_i\geq1$, and the logarithm
$\log(r_{i,\alpha}/\beta_{\alpha\mid i})$ in \eqref{eq:S-Jensen} equals}
\begin{equation}\label{eq:S-log-expanded}
 \log m+\log t_i+\log\alpha_i
 +\left(\frac2q-1\right)\log\mu_\alpha
 +\frac2p\sum_{j=1}^n\alpha_j\log t_j.
\end{equation}
{Multiplying \eqref{eq:S-log-expanded} by $\nu(i,\alpha)$, summing over
$i$ and $\alpha$, and using \eqref{eq:nu-marginals} yields}
\begin{align}
 \log\frac{S(P_t)}{|P|_q}
 &\geq\frac12\log m+\frac12\sum_{i=1}^n\sum_{|\alpha|=m}\nu(i,\alpha)\log\alpha_i
 +\left(\frac12-\frac mp\right)H(t)
 -\left(\frac1q-\frac12\right)H(\mu).
 \label{eq:S-entropy-expanded}
\end{align}
From \eqref{eq:upper-range-q-tau},
\[
 \frac1q-\frac12=\frac1{2m}-\frac1p,
 \qquad
 \frac12-\frac mp=m\left(\frac1{2m}-\frac1p\right).
\]
Thus \eqref{eq:S-entropy-expanded} becomes
\begin{equation}\label{eq:S-I-exact}
 \log\frac{S(P_t)}{|P|_q}
 \geq\frac12\log m+\frac12\sum_{i=1}^n\sum_{|\alpha|=m}\nu(i,\alpha)\log\alpha_i
 +\left(\frac1{2m}-\frac1p\right)I.
\end{equation}
Whenever $\nu(i,\alpha)>0$ we have $\alpha_i\geq1$. Hence
\[
 \sum_{i=1}^n\sum_{|\alpha|=m}
 \nu(i,\alpha)\log\alpha_i\geq0,
\]
and \eqref{eq:S-I-exact} gives
\[
 \log\frac{S(P_t)}{|P|_q}
 \geq
 \frac12\log m+
 \left(\frac1{2m}-\frac1p\right)I.
\]
Exponentiating and dividing by $\sqrt m$ gives directly
\begin{equation}\label{eq:S-lower}
 \frac{S(P_t)}{\sqrt m}
 \geq |P|_q
 \exp\!\left[\left(\frac1{2m}-\frac1p\right)I\right].
\end{equation}

\medskip
{We now combine \eqref{eq:P_t-q0-lower} and \eqref{eq:S-lower} to prove
\eqref{eq:numerical-transfer}.}\par\smallskip
{Raising \eqref{eq:P_t-q0-lower} to the power $1-\tau$ gives}
\[
 |P_t|_{q_m}^{\,1-\tau}
 \geq
 |P|_q^{\,1-\tau}
 \exp\!\bigl(-(1-\tau)I/p\bigr),
\]
{while raising \eqref{eq:S-lower} to the power $\tau$ gives}
\[
 \left(\frac{S(P_t)}{\sqrt m}\right)^\tau
 \geq
 |P|_q^\tau
 \exp\!\bigl(\tau(1/(2m)-1/p)I\bigr).
\]
Multiplying these two inequalities,
\begin{align*}
 |P_t|_{q_m}^{\,1-\tau}
 \left(\frac{S(P_t)}{\sqrt m}\right)^\tau
 &\geq
 |P|_q
 \exp\!\bigl([-(1-\tau)/p+\tau(1/(2m)-1/p)]I\bigr).
\end{align*}
{The two entropy contributions now cancel exactly. Indeed, since $\tau=2m/p$,}
{
\[
 -\frac{1-\tau}{p}+\tau\left(\frac1{2m}-\frac1p\right)
 =-\frac1p+\frac{\tau}{p}+\frac{\tau}{2m}-\frac{\tau}{p}
 =-\frac1p+\frac{\tau}{2m}=0.
\]
}
Thus the exponential factor is $1$, and
\[
 |P|_q
 \leq
 |P_t|_{q_m}^{\,1-\tau}
 \left(\frac{S(P_t)}{\sqrt m}\right)^\tau,
\]
which is \eqref{eq:numerical-transfer}.
\end{proof}

\subsection{Proof of Theorem A: Hardy--Littlewood transfer from \texorpdfstring{$p=\infty$}{p=infinity}}
\label{sec:proof-A}

Let $m\geq2$ and $2m\leq p\leq\infty$, and put
\[
 \frac1q=\frac1{q_m}-\frac1p,
 \qquad \tau=\frac{2m}{p}.
\]
Assume first that $2m\leq p<\infty$ and $P\in\Pol_m(\C^n)$ is nonzero. Let
$P_t$ be the rescaled polynomial supplied by
Lemma~\ref{lem:coefficient-transfer}. {Applying successively
\eqref{eq:numerical-transfer}, the definition of $D_m$ together with
\eqref{eq:S-upper}, and finally \eqref{eq:P_t-norm-transfer}, we obtain}
\begin{align*}
 |P|_q
 &\overset{\eqref{eq:numerical-transfer}}{\leq}
 |P_t|_{q_m}^{1-\tau}
 \left(\frac{S(P_t)}{\sqrt m}\right)^\tau\\
 &\overset{\eqref{eq:S-upper}}{\leq}
 (D_m\|P_t\|_\infty)^{1-\tau}
 \left[
 \left(\frac{\pi}{2}\sqrt m\,2^{(m-1)/2}\right)
 \|P_t\|_\infty
 \right]^\tau\\
 &=D_m^{1-\tau}
 \left(\frac{\pi}{2}\sqrt m\,2^{(m-1)/2}\right)^\tau
 \|P_t\|_\infty\\
 &\overset{\eqref{eq:P_t-norm-transfer}}{\leq}
 D_m^{1-\tau}
 \left(\frac{\pi}{2}\sqrt m\,2^{(m-1)/2}\right)^\tau
 \|P\|_p.
\end{align*}
Thus
\begin{equation}\label{eq:entropy-transfer}
 |P|_q\leq D_m^{\,1-\tau}
 \left(\frac{\pi}{2}\sqrt m\,2^{(m-1)/2}\right)^{\tau}\|P\|_p.
\end{equation}
For $P=0$ the estimate is trivial. {Since
\eqref{eq:entropy-transfer} holds for every $n$ and every
$P\in\Pol_m(\C^n)$, the definition of $D_{m,p}$ gives, for
$2m\leq p<\infty$,}
\begin{equation}\label{eq:entropy-H-bound}
 {D_{m,p}}\leq D_m^{\,1-2m/p}
 \left(\frac{\pi}{2}\sqrt m\,2^{(m-1)/2}\right)^{\,2m/p}.
\end{equation}
{At $p=\infty$, \eqref{eq:entropy-H-bound} reduces to the identity
$D_{m,\infty}=D_m$.}

{We now apply the quantitative transfer estimate to the sequence in Theorem~\ref{thm:A}.}
Since $p_m/m^2\to\infty$, we have $p_m\ge2m$ for all sufficiently large
$m$. Hence \eqref{eq:entropy-H-bound} gives
\[
 \frac{{D_{m,p_m}}}{D_m}
 \leq
 \left(\frac{\left(\frac{\pi}{2}\sqrt m\,2^{(m-1)/2}\right)}{D_m}\right)^{2m/p_m}
 \leq\left(\frac{\pi}{2}\sqrt m\,2^{(m-1)/2}\right)^{2m/p_m},
\]
{since the monomial $z_1^m$ gives $D_m\geq1$.} Moreover,
\[
 \log\left(\frac{\pi}{2}\sqrt m\,2^{(m-1)/2}\right)
 =\log(\pi/2)+\frac12\log m+\frac{m-1}{2}\log2
 =O(m).
\]
Thus
\[
 \log\left(\left(\frac{\pi}{2}\sqrt m\,2^{(m-1)/2}\right)^{2m/p_m}\right)
 =O\left(\frac{m^2}{p_m}\right)\longrightarrow0,
\]
and therefore
\[
 {D_{m,p_m}}\leq(1+o(1))D_m.
\]
{This completes the proof of Theorem~\ref{thm:A}.}

{
\begin{remark}\label{rem:support-transfer}
The coefficient-dependent rescaling preserves monomial support. Consequently,
the preceding argument applies without change to polynomial classes that are
stable under coordinate rescaling. In particular, for each fixed
$M\in\mathbb N$, the asymptotic contractivity result of Caro Montoya,
N\'u\~nez-Alarc\'on, and Serrano-Rodr\'iguez~\cite{CaroMontoyaNunezSerrano2026}
for polynomials whose nonzero monomials involve at most $M$ variables yields
the corresponding Hardy--Littlewood contractivity whenever
$p_m/m^2\to\infty$.
\end{remark}
}

\part{Lower-range estimates}

\section{Lower Hardy--Littlewood range}
\label{sec:intermediate-HL}
The lower-range estimates quantify the constants near $p=m$ and yield
both parts of Theorem~\ref{thm:B}.
{The next lemma is the lower-range analogue of Lemma~\ref{lem:coefficient-transfer}: it uses the same coefficient weights $\mu_\alpha$, the same numbers $t_i$, and the same coordinate rescaling. Here the conclusion is instead a lower bound for $F_q(Q)$, with $q=p/(p-m)$.}
\begin{lemma}\label{lem:lower-coefficient-reduction}
Let $m<p\leq2m$, let $q=p/(p-m)={q(m,p)}$, and let
\[
 P:\mathbb C^n\longrightarrow\mathbb C,\qquad
 P(z)=\sum_{|\alpha|=m}a_\alpha z^\alpha\neq0.
\]
For $|\alpha|=m$ and $1\leq i\leq n$, put
\[
 \mu_\alpha:=\frac{|a_\alpha|^q}{|P|_q^q},
 \qquad
 t_i:=\frac1m\sum_{|\alpha|=m}\alpha_i\mu_\alpha,
\]
and define
\[
 Q:\mathbb C^n\longrightarrow\mathbb C,\qquad
 Q(z_1,\ldots,z_n)
 :=P(t_1^{1/p}z_1,\ldots,t_n^{1/p}z_n).
\]
Finally, set
{
\begin{equation}\label{eq:Fq-definition}
 F_q(Q):={\sum_{i=1}^n|\partial_iQ|_q}
 =
 \sum_{i=1}^n
 \left(
 \sum_{|\alpha|=m}
 \alpha_i^q|a_\alpha|^q
 \prod_{\substack{1\leq j\leq n\\ \alpha_j>0}}
 t_j^{q\alpha_j/p}
 \right)^{1/q}.
\end{equation}
}
Then
{
\begin{equation}\label{eq:lower-Q-norm}
 \|Q\|_\infty\leq\|P\|_p,
\end{equation}}
and
\begin{equation}\label{eq:lower-coefficient-reduction}
 F_q(Q)\ge |P|_q\,m^{1-m/p}.
\end{equation}
\end{lemma}

\begin{proof}
Zero coefficients may be omitted. Since $\mu_\alpha\ge0$,
$\sum_{|\alpha|=m}\mu_\alpha=1$, and $|\alpha|=m$, we have
\[
 \sum_{i=1}^n t_i
 =\frac1m\sum_{|\alpha|=m}\mu_\alpha\sum_{i=1}^n\alpha_i=1.
\]
If $t_i=0$, then
$\sum_{|\alpha|=m}\alpha_i\mu_\alpha=0$, so $\alpha_i=0$ whenever
$\mu_\alpha>0$; such coordinates may therefore be omitted. {For $z\in\mathbb C^n$ with $|z_i|\le1$ for $1\le i\le n$,
\[
 \sum_{i=1}^n|t_i^{1/p}z_i|^p
 =\sum_{i=1}^n t_i|z_i|^p
 \le\sum_{i=1}^n t_i=1.
\]
Hence $(t_1^{1/p}z_1,\ldots,t_n^{1/p}z_n)$ belongs to the unit ball of
$\ell_p^n$, and therefore
\[
 |Q(z)|
 =\bigl|P(t_1^{1/p}z_1,\ldots,t_n^{1/p}z_n)\bigr|
 \le\|P\|_p.
\]
Taking the supremum over all $z\in\mathbb C^n$ with $|z_i|\le1$
for $1\le i\le n$ gives
\[
 \|Q\|_\infty
 =\sup_{|z_1|,\ldots,|z_n|\le1}|Q(z)|
 \le\|P\|_p,
\]
which is \eqref{eq:lower-Q-norm}.}

We now prove \eqref{eq:lower-coefficient-reduction}. For
$1\le i\le n$ and $|\alpha|=m$, define
\[
 \nu(i,\alpha):=\frac{\alpha_i\mu_\alpha}{m}.
\]
Then
\[
 \sum_{|\alpha|=m}\nu(i,\alpha)=t_i,
 \qquad
 \sum_{i=1}^n\nu(i,\alpha)=\mu_\alpha.
\]
For $t_i>0$, put
\[
 \beta_{\alpha\mid i}:=\frac{\nu(i,\alpha)}{t_i}
 =\frac{\alpha_i\mu_\alpha}{mt_i},
 \qquad
 \sum_{|\alpha|=m}\beta_{\alpha\mid i}=1,
\]
and set
\[
 s_i:=
 \left(
 \sum_{|\alpha|=m}\alpha_i^q|a_\alpha|^q
 \prod_{\substack{1\le j\le n\\ \alpha_j>0}}
 t_j^{q\alpha_j/p}
 \right)^{1/q},
 \qquad F_q(Q)=\sum_{i=1}^n s_i.
\]
Finally, write
\[
 H(t):=-\sum_{i=1}^n t_i\log t_i.
\]

We apply Jensen's inequality \eqref{eq:log-sum-jensen} first to the sum in
$i$. Since
\[
 \frac{F_q(Q)}{|P|_q}
 =\sum_{i=1}^n t_i\frac{s_i}{t_i|P|_q},
\]
we obtain
\begin{equation}\label{eq:lower-Jensen-outer}
 \log\frac{F_q(Q)}{|P|_q}
 \ge H(t)+\sum_{i=1}^n t_i\log\frac{s_i}{|P|_q}.
\end{equation}

Fix now $i$ with $t_i>0$. From the definition of $s_i$ and
$\mu_\alpha=|a_\alpha|^q/|P|_q^q$,
\[
 \frac{s_i^q}{|P|_q^q}
 =\sum_{|\alpha|=m}\alpha_i^q\mu_\alpha
 \prod_{\substack{1\le j\le n\\ \alpha_j>0}}
 t_j^{q\alpha_j/p}.
\]
{Since $\sum_{|\alpha|=m}\beta_{\alpha\mid i}=1$, Jensen's inequality
\eqref{eq:log-sum-jensen} applied to the sum on the right-hand side above,
with probability weights $\beta_{\alpha\mid i}$, gives directly
\[
 q\log\frac{s_i}{|P|_q}
 =\log\!\left(
 \sum_{|\alpha|=m}\alpha_i^q\mu_\alpha
 \prod_{\substack{1\le j\le n\\ \alpha_j>0}} t_j^{q\alpha_j/p}
 \right)
 \overset{\eqref{eq:log-sum-jensen}}{\ge}
 \sum_{|\alpha|=m}\beta_{\alpha\mid i}
 \log\!\left(
 \frac{\alpha_i^q\mu_\alpha}{\beta_{\alpha\mid i}}
 \prod_{\substack{1\le j\le n\\ \alpha_j>0}} t_j^{q\alpha_j/p}
 \right).
\]
}
Using $\beta_{\alpha\mid i}=\alpha_i\mu_\alpha/(mt_i)$, the logarithm in
the last sum is
\[
 \log m+\log t_i+(q-1)\log\alpha_i
 +\frac qp\sum_{j=1}^n\alpha_j\log t_j.
\]
Multiply by $t_i/q$, sum in $i$, and use
$t_i\beta_{\alpha\mid i}=\nu(i,\alpha)$ together with
\eqref{eq:lower-Jensen-outer}. This yields
\[
 \log\frac{F_q(Q)}{|P|_q}
 \ge
 \left(1-\frac1q-\frac mp\right)H(t)
 +\frac1q\log m
 +\left(1-\frac1q\right)
 \sum_{i=1}^n\sum_{|\alpha|=m}
 \nu(i,\alpha)\log\alpha_i.
\]
Since $q=p/(p-m)$,
\[
 1-\frac1q=\frac mp,
 \qquad
 1-\frac1q-\frac mp=0.
\]
Therefore
\[
 \log\frac{F_q(Q)}{|P|_q}
 \ge \frac1q\log m
 +\frac mp\sum_{i=1}^n\sum_{|\alpha|=m}
 \nu(i,\alpha)\log\alpha_i.
\]
Whenever $\nu(i,\alpha)>0$, one has $\alpha_i\ge1$, and hence
$\log\alpha_i\ge0$. Thus the double sum is nonnegative, and
\[
 \log\frac{F_q(Q)}{|P|_q}
 \ge\frac1q\log m
 =\left(1-\frac mp\right)\log m.
\]
Exponentiating proves
\[
 F_q(Q)\ge |P|_q\,m^{1-m/p},
\]
which is \eqref{eq:lower-coefficient-reduction}.
\end{proof}

\subsection{Scalar coefficient contraction in the lower range}\label{sec:scalar-contraction}

{For a polynomial $f$ on $\T$, write}
\[
 \|f\|_{H^1}:=\int_{\T}|f(\zeta)|\,d\mathfrak m_1(\zeta).
\]
{The classical comparison radius comes from Weissler's analytic
hypercontractivity \cite{Weissler1980}; polynomial forms are given in
\cite{Bayart2002,DefantMastylo2016}. Coefficient multipliers from $H^1$ into
$\ell_q$ are classical; see \cite{BlascoPavlovic2011}. The weighted $\ell_2$
{estimate with constant $1$ used in Lemma~\ref{lem:scalar-radius}} is the sharp Hardy--Littlewood--Carleman
inequality; Mateljevi\'c obtained the sharp constant $1$ \cite{Mateljevic1979}, and
Helson gave the multivariable extension \cite{Helson2006}. See also
\cite{CarandoDefantSevilla2015,BrevigEtAl2018,Kulikov2022}. {The coefficient-separation estimate in Lemma~\ref{lem:pairwise-separation} belongs to the classical line of coefficient inequalities for nonnegative trigonometric polynomials initiated by Egerv\'ary and Sz\'asz~\cite{EgervarySzasz1928} and developed in a two-coefficient form by Losonczi~\cite{Losonczi1992}. We record only the short $H^1$ reduction needed here.}
}

{
\begingroup
\begin{lemma}\label{lem:pairwise-separation}
Let $N\in\N$, let $2\le M\le N+1$, and let
\[
 f(z)=\sum_{r=0}^N a_rz^r
\]
be a polynomial. For every $0\le j<k\le M-1$,
\begin{equation}\label{eq:two-low-separation}
 |a_j|+|a_k|
 \le 2\cos\frac{\pi}{2M+1}\,\|f\|_{H^1}.
\end{equation}
\end{lemma}

\begin{proof}
The case $f=0$ is immediate. Assume $f\neq0$ and normalize $\|f\|_{H^1}=1$. By Riesz factorization, write
$g(z)=\sum_{r\ge0}g_rz^r$ and $h(z)=\sum_{r\ge0}h_rz^r$ with
$f=gh$ and $\|g\|_{H^2}=\|h\|_{H^2}=1$. For $0\le s\le M-1$, set
\[
 R_s:=\bigl(\mathbf 1_{\{r+t=s\}}\bigr)_{0\le r,t\le M}.
\]
With $u=(|g_0|,\ldots,|g_M|)$ and $v=(|h_0|,\ldots,|h_M|)$, one has
$\|u\|_2,\|v\|_2\le1$, and therefore
\[
 |a_j|+|a_k|\le u^T(R_j+R_k)v\le\|R_j+R_k\|.
\]
\minorfix{Restrict to coordinates $0,\ldots,k$, since the remaining rows and
columns vanish. Put $n=k+1$ and $d=k-j$. Let $J$ reverse these $n$
coordinates, and let $S$ be the shift $Se_r=e_{r+1}$ for $0\le r<n-1$,
with $Se_{n-1}=0$. Then
\[
 JR_k=I,\qquad JR_j=S^d,\qquad
 \|R_j+R_k\|=\|I+S^d\|.
\]
The residue classes modulo $d$ decompose $S^d$ into shifts $S_L$ on
$L\le n\le M$ coordinates. The matrix $(I+S_L)^*(I+S_L)$ is tridiagonal,
with off-diagonal entries $1$ and diagonal $(2,\ldots,2,1)$, interpreted
as $(1)$ when $L=1$. Its eigenvalues are
\[
 2+2\cos\frac{2r\pi}{2L+1},\qquad 1\le r\le L.
\]
Indeed, substituting $x_a=\sin((a+1)\theta)$ in the eigenvalue recurrence
gives the boundary conditions $x_{-1}=0$ and $x_L=-x_{L-1}$, hence
$\theta=2r\pi/(2L+1)$; see also Yueh~\cite[Theorem~1, p.~72]{Yueh2005}.
Consequently,
\[
 \|I+S_L\|=2\cos\frac{\pi}{2L+1}
 \le 2\cos\frac{\pi}{2M+1},
\]
and therefore}
\[
 \|R_j+R_k\|\le2\cos\frac{\pi}{2M+1}.
\]
This proves \eqref{eq:two-low-separation}.
\end{proof}
\endgroup
}

{
For $q\ge16$, define
\begin{equation}\label{eq:scalar-radius-parameters}
\begin{gathered}
 \mathcal A_q:=
 \left\{M\in\mathbb N:
 (q-2)\left(1-\cos\frac{\pi}{2M+1}\right)\ge\log q
 \right\},\\
 M_q:=\max\mathcal A_q,
 \qquad
 \rho_q:=\exp\left(-\frac{\log q}{qM_q}\right).
\end{gathered}
\end{equation}
The number $\rho_q$ is the contraction radius used in this lemma: the dilation
$f(z)\mapsto f(\rho_q z)$ replaces each coefficient $a_k$ by $\rho_q^k a_k$.
The set $\mathcal A_q$ is nonempty: for $M=1$ the defining inequality becomes
$(q-2)/2\ge\log q$, which holds for $q\ge16$. It is also finite,
because $1-\cos(\pi/(2M+1))\to0$ as $M\to\infty$; hence $M_q$ is well
defined.}

\begin{lemma}\label{lem:scalar-radius}
{For every $q\ge16$ and every polynomial
$ f(z)=\sum_{k=0}^N a_kz^k$, one has}
{
\begin{equation}\label{eq:scalar-radius-contraction}
 \left(\sum_{k=0}^N|a_k|^q\rho_q^{qk}\right)^{1/q}
 \le \|f\|_{H^1}.
\end{equation}
Furthermore,
\begin{equation}\label{eq:scalar-radius-asymptotic}
 -\log\rho_q=
 \left({\frac{2\sqrt2}{\pi}}+o(1)\right)
 \left(\frac{\log q}{q}\right)^{3/2}.
\end{equation}
}
\end{lemma}

\begin{proof}
The case $f=0$ is immediate. Assume $f\neq0$, normalize $\|f\|_{H^1}=1$, and set $a_k=0$ for $k>N$. {Thus, whenever needed, we may regard $f$ as a finite polynomial with arbitrarily many trailing zero coefficients.} By the Riesz factorization theorem
\cite[Section~2.4]{Duren1970}, choose $g,h\in H^2$ such that $f=gh$ and
$\|g\|_{H^2}=\|h\|_{H^2}=1$. Writing
$g=\sum_{r\ge0}g_rz^r$ and $h=\sum_{r\ge0}h_rz^r$, we have
$a_k=\sum_{r=0}^kg_rh_{k-r}$. {Applying Cauchy--Schwarz to this finite sum gives, for every $k\ge0$,
\[
 |a_k|^2
 =\left|\sum_{r=0}^k g_rh_{k-r}\right|^2
 \leq (k+1)\sum_{r=0}^k|g_r|^2|h_{k-r}|^2.
\]
Dividing by $k+1$ yields
\begin{equation}\label{eq:scalar-CS}
 \frac{|a_k|^2}{k+1}\leq\sum_{r=0}^k|g_r|^2|h_{k-r}|^2.
\end{equation}}

{The pointwise coefficient bound follows directly from the $H^1$ normalization: for every $k\ge0$,
{
\begin{equation}\label{eq:scalar-pointwise}
 |a_k|=\left|\int_{\T} f(\zeta)\overline{\zeta}^{\,k}\,d\mathfrak m_1(\zeta)\right|
 \le \int_{\T}|f(\zeta)|\,d\mathfrak m_1(\zeta)=1.
\end{equation}
}
{Summing \eqref{eq:scalar-CS} over $k\ge0$ and reindexing by $s=k-r$ gives
\[
 \sum_{k\ge0}\frac{|a_k|^2}{k+1}
 \le \sum_{k\ge0}\sum_{r=0}^k|g_r|^2|h_{k-r}|^2
 =\sum_{r,s\ge0}|g_r|^2|h_s|^2
 =\|g\|_{H^2}^2\|h\|_{H^2}^2=1.
\]
Together with \eqref{eq:scalar-pointwise}, this gives}}
\begin{equation}\label{eq:scalar-square-mass}
 |a_k|\le1,\qquad \sum_{k\ge0}\frac{|a_k|^2}{k+1}\le1.
\end{equation}

{Now take $M=M_q$, with $M_q$ defined in
\eqref{eq:scalar-radius-parameters}, and put}
\[
 \theta_M:=\frac{\pi}{2M+1},
 \qquad
 \lambda:=\frac{\log q}{M}.
\]
\begingroup
Since $\rho_q^q=e^{-\lambda}$, to prove
\eqref{eq:scalar-radius-contraction} it is enough to show
\begin{equation}\label{eq:weighted-target}
 \sum_{k\ge0}e^{-\lambda k}|a_k|^q\le1.
\end{equation}
If $M=1$, the set $\{0,\ldots,M-1\}$ has one element. If $M\ge2$,
Lemma~\ref{lem:pairwise-separation} applies (using trailing zero coefficients if
necessary). Hence, in every case, at most one index
$k\in\{0,\ldots,M-1\}$ can satisfy
\[
 |a_k|>\cos\theta_M.
\]
Call this index exceptional, if it exists. Thus the exceptional set is either
empty or consists of one index in $\{0,\ldots,M-1\}$. We first estimate every
nonexceptional coefficient, separating the two ranges $k<M$ and $k\ge M$.

\medskip
\noindent$\boldsymbol{\bullet}$\ \textbf{Low indices $0\le k<M$.}
If $k$ is nonexceptional, then $|a_k|\le\cos\theta_M$. By the definition of
$M=M_q$,
\[
 (\cos\theta_M)^{q-2}
 \le \exp\!\bigl(-(q-2)(1-\cos\theta_M)\bigr)
 \le\frac1q.
\]
Since $k+1\le M$,
\begin{equation}\label{eq:low-nonexceptional}
 e^{-\lambda k}|a_k|^q
 \le \frac1q|a_k|^2
 \le \frac Mq\frac{|a_k|^2}{k+1}.
\end{equation}

\medskip
\noindent$\boldsymbol{\bullet}$\ \textbf{Tail $k\ge M$.}
No exceptional index occurs in this range. Using only $|a_k|\le1$, set
$\phi(x):=(x+1)e^{-\lambda x}$. Since
$\lambda(M+1)>\log q>1$, the function $\phi$ is decreasing on $[M,\infty)$.
Hence
\[
 (k+1)e^{-\lambda k}\le(M+1)e^{-\lambda M}=\frac{M+1}{q},
\]
and therefore
\begin{equation}\label{eq:high-nonexceptional}
 e^{-\lambda k}|a_k|^q
 \le\frac{M+1}{q}\frac{|a_k|^2}{k+1},
 \qquad k\ge M.
\end{equation}

\medskip
\noindent$\boldsymbol{\bullet}$\ \textbf{Joining the two ranges.}
We claim that
\begin{equation}\label{eq:Mq-tail-balance}
 \frac Mq\le\frac{M+1}{q}\le1-e^{-\lambda}.
\end{equation}
Indeed, admissibility of $M=M_q$ and $1-\cos x\le x^2/2$ give
\[
 M^2\le\frac{\pi^2(q-2)}{8\log q}.
\]
If $\lambda\le1$, then $1-e^{-\lambda}\ge\lambda/2$. Moreover, the preceding bound implies
$4M^2\le q\log q$ for $q\ge16$, because
$\pi^2(q-2)\le2q(\log q)^2$. Hence
\[
 \frac{M+1}{q}\le\frac{2M}{q}
 \le\frac{\log q}{2M}=\frac\lambda2\le1-e^{-\lambda}.
\]
If $\lambda\ge1$, then $M\le\log q$, so
\[
 \frac{M+1}{q}\le\frac{\log q+1}{q}
 \le\frac{\log16+1}{16}<1-e^{-1}\le1-e^{-\lambda}.
\]
Combining \eqref{eq:low-nonexceptional},
\eqref{eq:high-nonexceptional}, and \eqref{eq:Mq-tail-balance}, we obtain the
single estimate
\begin{equation}\label{eq:all-nonexceptional}
 e^{-\lambda k}|a_k|^q
 \le(1-e^{-\lambda})\frac{|a_k|^2}{k+1}
\end{equation}
for every nonexceptional index $k$. Thus, if $E$ denotes the exceptional set,
\begin{equation}\label{eq:nonexceptional-sum}
 \sum_{k\notin E}e^{-\lambda k}|a_k|^q
 \le(1-e^{-\lambda})
 \sum_{k\notin E}\frac{|a_k|^2}{k+1}.
\end{equation}
\medskip
\noindent$\boldsymbol{\bullet}$\ \textbf{Adding the exceptional coefficient.}
There are only three possibilities.

\smallskip
\noindent$\circ$\ \emph{No exceptional index.}
By \eqref{eq:scalar-square-mass} and \eqref{eq:nonexceptional-sum},
\[
 \sum_{k\ge0}e^{-\lambda k}|a_k|^q
 \le(1-e^{-\lambda})\sum_{k\ge0}\frac{|a_k|^2}{k+1}\le1.
\]

\smallskip
\noindent$\circ$\ \emph{The exceptional index is $k_0\ge1$.}
Since $|a_{k_0}|\le1$,
\[
 e^{-\lambda k_0}|a_{k_0}|^q\le e^{-\lambda}.
\]
Together with \eqref{eq:nonexceptional-sum} and
\eqref{eq:scalar-square-mass}, this gives
\[
 \sum_{k\ge0}e^{-\lambda k}|a_k|^q
 \le e^{-\lambda}+(1-e^{-\lambda})\cdot1=1.
\]

\smallskip
\noindent$\circ$\ \emph{The exceptional index is $k_0=0$.}
Put $x:=|a_0|$. Since $a_0=g_0h_0$, the estimate
\eqref{eq:scalar-CS} summed over $k\ge1$ gives
\[
 \sum_{k\ge1}\frac{|a_k|^2}{k+1}\le1-x^2.
\]
Therefore, using $q\ge2$ and $0\le x\le1$,
\[
 \sum_{k\ge0}e^{-\lambda k}|a_k|^q
 \le x^q+(1-e^{-\lambda})(1-x^2)
 \le x^2+(1-e^{-\lambda})(1-x^2)\le1.
\]
Thus \eqref{eq:weighted-target} holds in every case. Since
$\rho_q^q=e^{-\lambda}$, this proves
\eqref{eq:scalar-radius-contraction}.
\endgroup

{
It remains to prove the asymptotic formula \eqref{eq:scalar-radius-asymptotic} for
$-\log\rho_q$. Put
\[
 J_q:=\left\lfloor\sqrt{\frac{q-2}{2\log q}}\right\rfloor.
\]
For $q\ge16$, $J_q\ge1$. With $y=\pi/(4J_q+2)\le\pi/6$, the monotonicity of
$\sin y/y$ on $(0,\pi/2]$ gives
\[
 \sin y\ge\frac3\pi y=\frac3{4J_q+2}\ge\frac1{2J_q},
 \qquad
 1-\cos\frac{\pi}{2J_q+1}
 =2\sin^2y\ge\frac1{2J_q^2}.
\]
Hence
\[
 (q-2)\left(1-\cos\frac{\pi}{2J_q+1}\right)
 \ge\frac{q-2}{2J_q^2}\ge\log q,
\]
so $J_q\in\mathcal A_q$ and $M_q\ge J_q\to\infty$. Set
\[
 X_q:=\sqrt{\frac{q-2}{\log q}}.
\]
The expansion
$1-\cos x=\frac12x^2(1+o(1))$ applies at
$x=\pi/(2M_q+1)$ and at $x=\pi/(2M_q+3)$. By maximality of $M_q$,
\[
 (q-2)\left(1-\cos\frac{\pi}{2M_q+1}\right)\ge\log q
 >
 (q-2)\left(1-\cos\frac{\pi}{2M_q+3}\right).
\]
Equivalently,
\[
 \frac{\pi^2{X_q}^2}{2(2M_q+1)^2}(1+o(1))
 \ge1>
 \frac{\pi^2{X_q}^2}{2(2M_q+3)^2}(1+o(1)).
\]
Thus
\[
 M_q=
 \left(\frac{\pi}{2\sqrt2}+o(1)\right)
 \sqrt{\frac{q-2}{\log q}}.
\]
{Substituting this into the definition of $\rho_q$ gives
\[
 -\log\rho_q=\frac{\log q}{qM_q}
 =\left(\frac{2\sqrt2}{\pi}+o(1)\right)\frac{(\log q)^{3/2}}{q\sqrt{q-2}}
 =\left(\frac{2\sqrt2}{\pi}+o(1)\right)\left(\frac{\log q}{q}\right)^{3/2},
\]
which is exactly \eqref{eq:scalar-radius-asymptotic}.}
}

\end{proof}

{The passage from one variable to the polytorus follows the Fubini--Minkowski iteration used by Helson; see \cite[proof of (3.2), pp.~89--90]{Helson2006}. We record the argument to keep the norm-one constant explicit.}

\begin{lemma}\label{lem:tensor-radius}
Let $q\ge2$ and let $\rho\in(0,1]$ satisfy
\[
 \left(\sum_{k\ge0}|c_k|^q\rho^{qk}\right)^{1/q}
 \le \left\|\sum_{k\ge0}c_kz^k\right\|_{L^1(\T)}
\]
for every {one-variable polynomial}. Then for every {polynomial in $n$ variables},
\[
 F:\mathbb C^n\longrightarrow\mathbb C,\qquad
 F(z)=\sum_{\alpha\in\mathbb N_0^n}a_\alpha z^\alpha,
\]
\[
 \left(\sum_{\alpha\in\mathbb N_0^n}
 |a_\alpha|^q\rho^{q|\alpha|}\right)^{1/q}
 \le\|F\|_{L^1(\mathbb T^n)}.
\]
In particular, for $d$-homogeneous $F$,
$|F|_q\le\rho^{-d}\|F\|_{L^1}$.
\end{lemma}

\begin{proof}
We argue by induction on $n$. The case $n=1$ is precisely the scalar
hypothesis. Assume that the assertion has been proved in $n-1$ variables and
write
\[
 F(z,w)=\sum_{j\geq0}F_j(z)w^j
       =\sum_{\alpha\in\mathbb N_0^{n-1}}\sum_{j\geq0}
         a_{\alpha,j}z^\alpha w^j,
 \qquad z\in\T^{n-1},\quad w\in\T .
\]

For each fixed $j$, the induction hypothesis applied to $F_j$ gives
\[
 \left(
 \sum_{\alpha\in\mathbb N_0^{n-1}}
 |a_{\alpha,j}|^q\rho^{q|\alpha|}
 \right)^{1/q}
 \le \|F_j\|_{L^1(\T^{n-1})}.
\]
Multiplying by $\rho^j$, raising to the $q$-th power, and summing in $j$ gives
{
\[
 \left(
 \sum_{\alpha\in\mathbb N_0^{n-1}}\sum_{j\geq0}
 |a_{\alpha,j}|^q\rho^{q(|\alpha|+j)}
 \right)^{1/q}
 \le
 \left(
 \sum_{j\geq0}\rho^{qj}\|F_j\|_{L^1(\T^{n-1})}^q
 \right)^{1/q}.
\]
}
 {Let $\mathfrak m_{n-1}$ denote normalized Haar measure on $\T^{n-1}$.
\par
Since $\rho^j\|F_j\|_{L^1(\T^{n-1})}=\int_{\T^{n-1}}\rho^j|F_j(z)|\,d\mathfrak m_{n-1}(z)$,
Minkowski's integral inequality yields}
{
\[
 \left(\sum_{j\geq0}\rho^{qj}\|F_j\|_{L^1(\T^{n-1})}^q\right)^{1/q}
 \le
 \int_{\T^{n-1}}\left(\sum_{j\geq0}\rho^{qj}|F_j(z)|^q\right)^{1/q}
 d\mathfrak m_{n-1}(z).
\]
}
{Now fix $z\in\T^{n-1}$. The one-variable polynomial
$w\mapsto F(z,w)=\sum_{j\ge0}F_j(z)w^j$ has coefficients $F_j(z)$, so the scalar contraction gives}
\[
 \left(
 \sum_{j\geq0}\rho^{qj}|F_j(z)|^q
 \right)^{1/q}
 \le
 \int_{\T}|F(z,w)|\,d\mathfrak m_1(w),
\]
where $\mathfrak m_1$ is normalized Haar measure on $\T$.  Integrating this inequality
in $z$ and using Fubini's theorem, we conclude that
{
\[
 \left(\sum_{\alpha\in\mathbb N_0^{n-1}}\sum_{j\geq0}
 |a_{\alpha,j}|^q\rho^{q(|\alpha|+j)}\right)^{1/q}
 \le \int_{\T^{n-1}}\int_{\T}|F(z,w)|\,d\mathfrak m_1(w)\,d\mathfrak m_{n-1}(z)
 =\|F\|_{L^1(\T^n)}.
\]
}
This proves the induction step without accumulating any multiplicative
constant as the number of variables increases.

Finally, if $F$ is $d$-homogeneous, then $|\alpha|=d$ for every nonzero
coefficient. Hence
{
\[
 \left(\sum_{|\alpha|=d}|a_\alpha|^q\rho^{qd}\right)^{1/q}
 =\rho^d|F|_q\le\|F\|_{L^1(\T^n)}.
\]
}
Since $\rho^d>0$, this is equivalent to
{$|F|_q\le\rho^{-d}\|F\|_{L^1(\T^n)}$.}

\end{proof}

\begin{proposition}\label{prop:technical-lower}
For $m<p\le2m$, put $q=p/(p-m)$ and define
\[
 \psi(q):=
 \begin{cases}
 \dfrac12\log\dfrac{q}{q-1},&2\le q<16,\\[2mm]
 \min\left\{\dfrac12\log\dfrac{q}{q-1},-\log\rho_q\right\},&q\ge16.
 \end{cases}
\]
{
\begin{enumerate}[label=\textup{(\roman*)}]
\item One has
\[
 {D_{m,p}}\le \frac{\pi}{2}m^{m/p}\exp\bigl((m-1)\psi(q)\bigr).
\]
\item As $q\to\infty$,
\[
 \psi(q)=\left(\frac{2\sqrt2}{\pi}+o(1)\right)
 \left(\frac{\log q}{q}\right)^{3/2}.
\]
\end{enumerate}
}
\end{proposition}

\begin{proof}
Let $n\in\N$ and $0\ne P\in\Pol_m(\C^n)$. Apply
Lemma~\ref{lem:lower-coefficient-reduction} to $P$, and denote the resulting
polynomial by $Q$. {We shall use both
\eqref{eq:lower-coefficient-reduction} and \eqref{eq:lower-Q-norm}; before
combining them, we identify the coefficient norm appearing in
\eqref{eq:Fq-definition}. From the definition of $Q$ in
Lemma~\ref{lem:lower-coefficient-reduction},
\[
 Q(z)=\sum_{|\alpha|=m}a_\alpha
 \left(\prod_{\substack{1\le j\le n\\ \alpha_j>0}}
 t_j^{\alpha_j/p}\right)z^\alpha.
\]
Differentiating this finite sum term by term with respect to $z_i$ gives
\[
 \partial_iQ(z)
 =\sum_{\substack{|\alpha|=m\\ \alpha_i\ge1}}
 \alpha_i a_\alpha
 \left(\prod_{\substack{1\leq j\leq n\\ \alpha_j>0}}
 t_j^{\alpha_j/p}\right)z^{\alpha-e_i}.
\]
For each $1\leq i\leq n$, set
\[
 A_i:=
 \left(
 \sum_{|\alpha|=m}
 \alpha_i^q|a_\alpha|^q
 \prod_{\substack{1\leq j\leq n\\ \alpha_j>0}}
 t_j^{q\alpha_j/p}
 \right)^{1/q}.
\]
By the coefficient formula for $\partial_iQ$, one has $A_i=|\partial_iQ|_q$; hence \eqref{eq:Fq-definition} reads
\[
 F_q(Q)=\sum_{i=1}^n A_i.
\]
Since $\partial_iQ$ is $(m-1)$-homogeneous, Lemma~\ref{lem:tensor-radius}, applied with exponent $q$ and any admissible scalar contraction radius $\rho$, gives
\[
 A_i=|\partial_iQ|_q
 \le \rho^{-(m-1)}\|\partial_iQ\|_{L^1(\T^n)}.
\]
Summing over $i$ therefore yields
\[
 F_q(Q)
 \le \rho^{-(m-1)}
 \sum_{i=1}^n\|\partial_iQ\|_{L^1(\T^n)}.
\]
}
{Next, integrating \eqref{eq:derivative-l1-pointwise} over $\T^n$ gives}
\begin{equation}\label{eq:lower-derivative-L1}
 \sum_{i=1}^n\|\partial_iQ\|_{L^1(\T^n)}
 \le \frac{\pi}{2}\,m\,\|Q\|_\infty.
\end{equation}
The preceding estimates now combine in a single chain:
\[
\begin{aligned}
 |P|_q\,m^{1-m/p}
 &\overset{\eqref{eq:lower-coefficient-reduction}}{\le} F_q(Q)\\
 &\le \rho^{-(m-1)}\sum_{i=1}^n\|\partial_iQ\|_{L^1(\T^n)}\\
 &\overset{\eqref{eq:lower-derivative-L1}}{\le}
 \frac{\pi}{2}\,m\,\rho^{-(m-1)}\|Q\|_\infty\\
 &\overset{\eqref{eq:lower-Q-norm}}{\le}
 \frac{\pi}{2}\,m\,\rho^{-(m-1)}\|P\|_p.
\end{aligned}
\]
Dividing by $m^{1-m/p}$ gives
\begin{equation}\label{eq:technical-lower-radius}
 |P|_q
 \le
 \frac{\pi}{2}\,
 m^{m/p}\rho^{-(m-1)}\|P\|_p.
\end{equation}

We now choose the admissible radius in \eqref{eq:technical-lower-radius}. For every $q\ge2$, put
\[
 q':=\frac{q}{q-1},
 \qquad
 r_q:=\sqrt{\frac1{q'}}
      =\sqrt{1-\frac1q}.
\]
Weissler's analytic hypercontractive inequality
\cite[Corollary~2.1]{Weissler1980} gives, for every {one-variable polynomial $f$},
\[
 \|f(r_q\,\cdot)\|_{L^{q'}(\T)}
 \le \|f\|_{L^1(\T)}.
\]
If
\[
 f(z)=\sum_{k\ge0}c_kz^k,
\]
then the Fourier coefficients of $f(r_q\,\cdot)$ are $c_kr_q^k$ for
$k\ge0$. Since $1\le q'\le2$ and $(q')'=q$, the Hausdorff--Young inequality on $\T$
\cite[Theorem~IV.2.1]{Katznelson2004}, combined with Weissler's estimate, gives
{
\[
 \left(\sum_{k\ge0}|c_k|^q r_q^{qk}\right)^{1/q}
 \le \|f(r_q\,\cdot)\|_{L^{q'}(\T)}
 \le \|f\|_{L^1(\T)}.
\]
Thus $r_q$ is an admissible contraction radius for every $q\ge2$. When $q\ge16$, Lemma~\ref{lem:scalar-radius} supplies a second admissible radius, $\rho_q$. Since the factor $\rho^{-(m-1)}$ in \eqref{eq:technical-lower-radius} decreases as $\rho$ increases, the best estimate obtained from the available radii uses the larger one. We therefore set
\[
 R_q:=\begin{cases}
 r_q,&2\le q<16,\\[1mm]
 \max\{r_q,\rho_q\},&q\ge16.
 \end{cases}
\]
Then
\[
 -\log R_q=\psi(q),
\]
with $\psi$ as defined in the statement. Substituting $\rho=R_q$ into
\eqref{eq:technical-lower-radius} gives
\[
 {D_{m,p}}
 \le
 \frac{\pi}{2}m^{m/p}\exp\bigl((m-1)\psi(q)\bigr),
\]
which proves \textup{(i)}.

{We now prove \textup{(ii)}.} By
\eqref{eq:scalar-radius-asymptotic},
\[
 -\log\rho_q
 =\left(\frac{2\sqrt2}{\pi}+o(1)\right)
 \left(\frac{\log q}{q}\right)^{3/2},
\]
whereas
\[
 \frac12\log\frac{q}{q-1}=\frac1{2q}+O(q^{-2}).
\]
Their ratio satisfies
\[
 \frac{-\log\rho_q}{\frac12\log(q/(q-1))}
 =O\!\left(\frac{(\log q)^{3/2}}{\sqrt q}\right)\longrightarrow0.
\]
Hence $-\log\rho_q<\frac12\log(q/(q-1))$ for all sufficiently large $q$.
Therefore the minimum defining $\psi(q)$ is eventually $-\log\rho_q$, and
\[
 \psi(q)
 =\left(\frac{2\sqrt2}{\pi}+o(1)\right)
 \left(\frac{\log q}{q}\right)^{3/2},
\]
{which proves \textup{(ii)}.}
}
\end{proof}

\subsection{Proof of Theorem B\textup{(i)}}

{

Fix $a>0$. It is enough to prove the uniform estimate
}
\[
 \sup_{m<p\le m+a(m/\log m)^{1/3}}
 \frac{\log{D_{m,p}}}{\log m}
 \le 1+\frac{8}{3\pi\sqrt3}a^{3/2}+o(1).
\]
{
For $m<p\le m+a(m/\log m)^{1/3}$, set
}
\[
 h:=p-m,
 \qquad
 x:=\frac{h}{m+h}=\frac1q.
\]
{
Then
}
\[
 0<h\le a\left(\frac{m}{\log m}\right)^{1/3}=o(m),
 \qquad
 0<x\le\frac hm
 \le a\,m^{-2/3}(\log m)^{-1/3}\longrightarrow0,
\]
{
uniformly throughout the strip. Hence $m<p<2m$ for all sufficiently large
$m$, and Proposition~\ref{prop:technical-lower}\textup{(i)} applies. Since
$q=1/x\to\infty$ uniformly, Proposition~\ref{prop:technical-lower}\textup{(ii)}
may be used uniformly as well. We therefore obtain
}
\[
 \log{D_{m,p}}
 \le \log(\pi/2)+\frac mp\log m
 +\left(\frac{2\sqrt2}{\pi}+o(1)\right)
 m[x\log(1/x)]^{3/2}.
\]
{
Here $\log(\pi/2)=o(\log m)$ and, since $p>m$,
\[
 \frac mp\log m\leq\log m.
\]
Hence the first two terms on the right-hand side are bounded by
$(1+o(1))\log m$. Thus, to prove the desired uniform estimate, it remains only
to control
\[
 m[x\log(1/x)]^{3/2}
\]
uniformly in $p$. For large $m$, every admissible $x$ lies in $(0,e^{-1})$, and
}
\[
 \frac{d}{dx}\bigl[x\log(1/x)\bigr]=\log(1/x)-1>0
 \qquad(0<x<e^{-1}).
\]
{
Moreover, $x=h/(m+h)$ is increasing in $h$. Thus the largest value of
$x\log(1/x)$ on the strip is attained at
}
\[
 h_m:=a\left(\frac{m}{\log m}\right)^{1/3}.
\]
{
Writing $x_m:=h_m/(m+h_m)$, we have
}
\[
 x_m=a m^{-2/3}(\log m)^{-1/3}(1+o(1)),
 \qquad
 \log(1/x_m)=\left(\frac23+o(1)\right)\log m,
\]
{
and hence
}
\[
 x_m\log(1/x_m)
 =\left(\frac{2a}{3}+o(1)\right)
 m^{-2/3}(\log m)^{2/3}.
\]
{
Consequently, uniformly for all $p$ in the strip,
}
\[
 \left(\frac{2\sqrt2}{\pi}+o(1)\right)
 m[x\log(1/x)]^{3/2}
 \le
 \left(\frac{8}{3\pi\sqrt3}a^{3/2}+o(1)\right)\log m.
\]
{
Combining this estimate with the preceding reduction gives
}
\[
 \sup_{m<p\le m+a(m/\log m)^{1/3}}
 \log{D_{m,p}}
 \le
 \left(1+\frac{8}{3\pi\sqrt3}a^{3/2}+o(1)\right)\log m.
\]
{
Dividing by $\log m$ and taking $\limsup$ proves Theorem~\ref{thm:B}\textup{(i)}.
}

\subsection{Proof of Theorem B\textup{(ii)}}

{
{The proof follows the same argument as Theorem~\ref{thm:B}\textup{(i)} on the wider scale $p-m\lesssim(m\log m)^{1/3}$; only the final normalization changes.}
Fix $a>0$.  We show directly that
\[
 \sup_{m<p\le m+a(m\log m)^{1/3}}
 \log{D_{m,p}}
 \le
 \left(\frac{8}{3\pi\sqrt3}a^{3/2}+o(1)\right)(\log m)^2.
\]
For $p$ in this strip set
\[
 {h:=p-m},
 \qquad
 x:=\frac{{h}}{m+{h}}=\frac1q.
\]
Now
\[
 0<x\le\frac{{h}}m
 \le a\,m^{-2/3}(\log m)^{1/3}\longrightarrow0,
\]
{Since $h\le a(m\log m)^{1/3}=o(m)$, one also has $m<p<2m$ on the whole strip for all sufficiently large $m$.} {Hence Proposition~\ref{prop:technical-lower}\textup{(i)} applies throughout the strip. Moreover, $q=1/x\to\infty$ uniformly, so Proposition~\ref{prop:technical-lower}\textup{(ii)} may be used uniformly as well. Thus}
\[
 \log{D_{m,p}}
 \le \log(\pi/2)+\frac mp\log m
 +\left(\frac{2\sqrt2}{\pi}+o(1)\right)
 m[x\log(1/x)]^{3/2}.
\]
As in the proof of Theorem~\ref{thm:B}\textup{(i)}, for large $m$ all admissible values
of $x$ lie in $(0,e^{-1})$, where $x\log(1/x)$ is increasing.  Since
$x={h}/(m+{h})$ is also increasing in ${h}$, we may bound the correction term by
its value at
\[
 {h_m:=a(m\log m)^{1/3}}.
\]
For this value,
\[
 x=a m^{-2/3}(\log m)^{1/3}(1+o(1)),
 \qquad
 \log(1/x)=\left(\frac23+o(1)\right)\log m,
\]
and hence
\[
 x\log(1/x)
 =\left(\frac{2a}{3}+o(1)\right)
 m^{-2/3}(\log m)^{4/3}.
\]
It follows that, uniformly for all $p$ in the strip,
\[
 \left(\frac{2\sqrt2}{\pi}+o(1)\right)
 m[x\log(1/x)]^{3/2}
 \le
 \left(\frac{8}{3\pi\sqrt3}a^{3/2}+o(1)\right)(\log m)^2.
\]
Finally,
\[
 \log(\pi/2)+\frac mp\log m
 \le \log(\pi/2)+\log m
 =o((\log m)^2).
\]
Therefore
\[
 \sup_{m<p\le m+a(m\log m)^{1/3}}
 \log{D_{m,p}}
 \le
 \left(\frac{8}{3\pi\sqrt3}a^{3/2}+o(1)\right)(\log m)^2.
\]
Dividing by $(\log m)^2$ and taking $\limsup$ completes the proof of
Theorem~\ref{thm:B}\textup{(ii)}.
}

\part{Multiplicity patterns and endpoint stability}

\section{Mixed coefficient estimates}\label{sec:mixed}

{The multiplicity-pattern decomposition requires coefficient estimates that keep one coordinate block separate from the others. {The case $t=2$ of Lemma~\ref{lem:mixed-source}} is the classical mixed $(\ell_1,\ell_2)$ Littlewood estimate; see, for instance, Defant and Sevilla-Peris~\cite{DefantSevilla2009}, and for the optimal complex constants see \cite{CavalcanteNunezPellegrinoRueda2022}. The range $t\ge2$ follows here by combining that estimate with coefficientwise control. We include the short argument to keep the Steinhaus constant explicit.}

{Throughout this section, if
$T:E_1\times\cdots\times E_s\longrightarrow\C$ is an $s$-linear form, then
\[
 \|T\|:=\sup\bigl\{|T(x^{(1)},\ldots,x^{(s)})|:
 \|x^{(k)}\|_{E_k}\le1\ (1\le k\le s)\bigr\}
\]
denotes its operator norm.}

\begin{lemma}\label{lem:mixed-source}
Let $s\geq2$ and let $N_1,\ldots,N_s\in\N$. Set
\[
 \mathcal I:=\prod_{k=1}^s\{1,\ldots,N_k\},\qquad
 \mathcal I_{i,a}:=\{\boldsymbol j\in\mathcal I:j_i=a\}.
\]
For every $t\in[2,\infty)$ and every $s$-linear form
\[
 T:\ell_\infty^{N_1}\times\cdots\times\ell_\infty^{N_s}\longrightarrow\C,
\]
we have
\begin{equation}\label{eq:mixed-source}
 \sum_{a=1}^{N_i}
 \left(\sum_{\boldsymbol j\in\mathcal I_{i,a}}
       |T(e_{j_1},\ldots,e_{j_s})|^t\right)^{1/t}
 \leq\left(\frac2{\sqrt\pi}\right)^{2(s-1)/t}\|T\|,
 \qquad 1\leq i\leq s.
\end{equation}
\end{lemma}

\begin{proof}
Fix $i\in\{1,\ldots,s\}$. For each $k\neq i$, write
$z^{(k)}=(z^{(k)}_1,\ldots,z^{(k)}_{N_k})\in\T^{N_k}$. We collect the
$s-1$ coordinate blocks different from the $i$th block in the product space
\[
 \mathcal T_i:=\prod_{\substack{k=1\\k\neq i}}^s\T^{N_k},
 \qquad
 {\eta_i:=\bigotimes_{\substack{k=1\\k\neq i}}^s\mathfrak m_{N_k}.}
\]
Thus a point of $\mathcal T_i$ is a tuple
$z=(z^{(k)})_{k\neq i}$, and $\eta_i$ is the normalized product Haar
measure on this space.
For $a\in\{1,\ldots,N_i\}$ and $z=(z^{(k)})_{k\neq i}\in\mathcal T_i$,
define $F_{i,a}:\mathcal T_i\to\C$ by
\begin{equation}\label{eq:F-definition}
 F_{i,a}(z):=T(z^{(1)},\ldots,z^{(i-1)},e_a,z^{(i+1)},\ldots,z^{(s)})
 =\sum_{\boldsymbol j\in\mathcal I_{i,a}}T(e_{j_1},\ldots,e_{j_s})
 \prod_{\substack{k=1\\k\neq i}}^s z^{(k)}_{j_k}.
\end{equation}
Successive applications of the one-dimensional Khintchine inequality for
Steinhaus variables give
\[
 \left(\sum_{j=1}^{N}|c_j|^2\right)^{1/2}
 \leq \frac{2}{\sqrt\pi}\int_{\T^N}
       \left|\sum_{j=1}^{N}c_j z_j\right|\,d\mathfrak m_N(z),
 \qquad c_1,\ldots,c_N\in\C,
\]
whose optimal constant is $2/\sqrt\pi$; see \cite[Theorem~B]{Sawa1985}.
{\displaywidowpenalty=10000
We apply this estimate one coordinate block at a time. If $s=2$, there is only one block different from $i$, and the preceding inequality gives the required estimate directly. Assume $s\ge3$. To make the iteration explicit, consider two blocks $k,\ell\neq i$. For a coefficient array $(c_{u,v})$, applying the Steinhaus inequality in the $\ell$th block, then Minkowski's integral inequality, and finally the same Steinhaus inequality in the $k$th block gives
\minorfix{
\[
\begin{aligned}
 \left(\sum_{u,v}|c_{u,v}|^2\right)^{1/2}
 &\leq \frac2{\sqrt\pi}\int_{\T^{N_\ell}}
 \left(\sum_u\left|\sum_v c_{u,v}z_v^{(\ell)}\right|^2\right)^{1/2}d\mathfrak m_{N_\ell}
 \\
 &\leq \left(\frac2{\sqrt\pi}\right)^2
 \int_{\T^{N_k}}\!\int_{\T^{N_\ell}}
 \left|\sum_{u,v}c_{u,v}z_u^{(k)}z_v^{(\ell)}\right|d\mathfrak m_{N_\ell}d\mathfrak m_{N_k}.
\end{aligned}
\]}
Thus each additional block contributes one factor $2/\sqrt\pi$; repeating this argument over the $s-1$ blocks indexed by $k\neq i$ and using Fubini yields}
\begin{equation}\label{eq:F-l2}
 \left(\sum_{\boldsymbol j\in\mathcal I_{i,a}}
 |T(e_{j_1},\ldots,e_{j_s})|^2\right)^{1/2}
 \leq\left(\frac2{\sqrt\pi}\right)^{s-1}
       \int_{\mathcal T_i}|F_{i,a}(z)|\,d\eta_i(z).
\end{equation}
Fix $\boldsymbol j=(j_1,\ldots,j_s)\in\mathcal I_{i,a}$. Multiplying
\eqref{eq:F-definition} by
$\prod_{k\neq i}\overline{z^{(k)}_{j_k}}$ and integrating with respect
to the product measure $\eta_i$ gives
\begin{align*}
 &\int_{\mathcal T_i}F_{i,a}(z)
   \prod_{\substack{k=1\\k\neq i}}^s
   \overline{z^{(k)}_{j_k}}\,d\eta_i(z)\\
 &=\sum_{\boldsymbol \ell\in\mathcal I_{i,a}}
 T(e_{\ell_1},\ldots,e_{\ell_s})
 \int_{\mathcal T_i}
 \prod_{\substack{k=1\\k\neq i}}^s
 z^{(k)}_{\ell_k}\overline{z^{(k)}_{j_k}}\,d\eta_i(z)\\
 &=\sum_{\boldsymbol \ell\in\mathcal I_{i,a}}
 T(e_{\ell_1},\ldots,e_{\ell_s})
 \prod_{\substack{k=1\\k\neq i}}^s
 \int_{\T^{N_k}}z^{(k)}_{\ell_k}
 \overline{z^{(k)}_{j_k}}\,d\mathfrak m_{N_k}(z^{(k)}).
\end{align*}
By the standard orthogonality of the characters of $\mathbb T$, for every $k\neq i$,
\[
 \int_{\T^{N_k}}z^{(k)}_{\ell_k}
 \overline{z^{(k)}_{j_k}}\,d\mathfrak m_{N_k}(z^{(k)})
 =\begin{cases}1,&\ell_k=j_k,\\0,&\ell_k\neq j_k.\end{cases}
\]
Since $\ell_i=j_i=a$ for $\boldsymbol \ell,\boldsymbol j\in\mathcal I_{i,a}$,
only the term $\boldsymbol \ell=\boldsymbol j$ remains. Hence
\begin{equation}\label{eq:coefficient-extraction}
 T(e_{j_1},\ldots,e_{j_s})
 =\int_{\mathcal T_i}F_{i,a}(z)
   \prod_{\substack{k=1\\k\neq i}}^s
   \overline{z^{(k)}_{j_k}}\,d\eta_i(z).
\end{equation}
Because every factor in the product has modulus one,
\begin{equation}\label{eq:F-maximum}
 \max_{\boldsymbol j\in\mathcal I_{i,a}}
 |T(e_{j_1},\ldots,e_{j_s})|
 \leq\int_{\mathcal T_i}|F_{i,a}(z)|\,d\eta_i(z).
\end{equation}
For $t>2$, the identity $|b|^t=|b|^{t-2}|b|^2$ gives
\begin{align*}
 &\left(\sum_{\boldsymbol j\in\mathcal I_{i,a}}
 |T(e_{j_1},\ldots,e_{j_s})|^t\right)^{1/t}\\
 &\quad\leq
 \left(\max_{\boldsymbol j\in\mathcal I_{i,a}}
       |T(e_{j_1},\ldots,e_{j_s})|\right)^{1-2/t}
 \left(\sum_{\boldsymbol j\in\mathcal I_{i,a}}
       |T(e_{j_1},\ldots,e_{j_s})|^2\right)^{1/t}\\
 &\quad\leq\left(\frac2{\sqrt\pi}\right)^{2(s-1)/t}
             \int_{\mathcal T_i}|F_{i,a}(z)|\,d\eta_i(z).
\end{align*}
For $t=2$, the same conclusion is \eqref{eq:F-l2}.
{For each fixed $z\in\mathcal T_i$, complex $\ell_\infty$ duality and the definition of $F_{i,a}$ give
\[
 \sum_{a=1}^{N_i}|F_{i,a}(z)|
 =\sup_{\|w\|_\infty\leq1}|T(z^{(1)},\ldots,z^{(i-1)},w,z^{(i+1)},\ldots,z^{(s)})|
 \leq\|T\|.
\]
Therefore
{
\[
 \sum_{a=1}^{N_i}\left(\sum_{\boldsymbol j\in\mathcal I_{i,a}}|T(e_{j_1},\ldots,e_{j_s})|^t\right)^{1/t}
 \leq \left(\frac2{\sqrt\pi}\right)^{2(s-1)/t}\int_{\mathcal T_i}\sum_{a=1}^{N_i}|F_{i,a}(z)|\,d\eta_i(z)
 \leq \left(\frac2{\sqrt\pi}\right)^{2(s-1)/t}\|T\|,
\]
}
which is \eqref{eq:mixed-source}.}
\end{proof}

{The next lemma transfers a mixed coefficient estimate from products of
$\ell_\infty$ spaces to anisotropic products
$\ell_{r_1}\times\cdots\times\ell_{r_s}$, preserving the same constant.}

\begin{lemma}\label{lem:specialized-transfer}
Let $s,N_1,\ldots,N_s\in\N$, let $r_1,\ldots,r_s\in(1,\infty)$, and set
\[
 \sigma:=\sum_{k=1}^s\frac1{r_k}<1,
 \qquad \kappa:=(1-\sigma)^{-1}.
\]
{Set
\[
 \mathcal I:=\prod_{k=1}^s\{1,\ldots,N_k\},\qquad
 \mathcal I_{i,a}:=\{\boldsymbol j\in\mathcal I:j_i=a\}.
\]}
Let $t\geq\kappa$ and $C>0$. For $M\in\N$, set
\[
 \mathcal I^M:=\{1,\ldots,M\}^s,\qquad
 \mathcal I^M_{i,a}:=\{\boldsymbol j\in\mathcal I^M:j_i=a\}.
\]
Assume that, for every $M\in\N$, every $s$-linear form
\[
 T:(\ell_\infty^M)^s\longrightarrow\C
\]
satisfies, for $1\leq i\leq s$,
\begin{equation}\label{eq:transfer-source}
 \sum_{a=1}^{M}
 \left(\sum_{\boldsymbol j\in\mathcal I^M_{i,a}}
 |T(e_{j_1},\ldots,e_{j_s})|^t\right)^{1/t}
 \leq C\|T\|.
\end{equation}
Then every $s$-linear form
\[
 B:\ell_{r_1}^{N_1}\times\cdots\times\ell_{r_s}^{N_s}\longrightarrow\C
\]
satisfies
\begin{equation}\label{eq:specialized-transfer}
 \left[\sum_{a=1}^{N_i}
 \left(\sum_{\boldsymbol j\in\mathcal I_{i,a}}
 |B(e_{j_1},\ldots,e_{j_s})|^t\right)^{\kappa/t}
 \right]^{1/\kappa}
 \leq C\|B\|,
 \qquad 1\leq i\leq s.
\end{equation}
\end{lemma}

\begin{proof}
{
For $s=1$, the hypothesis \eqref{eq:transfer-source} with $M=1$
gives $C\ge1$. Since $\kappa=r_1/(r_1-1)$, duality gives
\[
 \left(\sum_{a=1}^{N_1}|B(e_a)|^\kappa\right)^{1/\kappa}
 =\|B\|\le C\|B\|,
\]
which is \eqref{eq:specialized-transfer}. We now assume $s\ge2$.
}
The transfer result in \auditfix{\cite[Lemma~2.1(1)]{AlbuquerqueEtAlUniform}} is formulated
for a common finite dimension.  Set
\[
 N:=\max_{1\leq k\leq s}N_k
\]
and define
{$\iota_k:\ell_{r_k}^{N_k}\to\ell_{r_k}^{N}$ and
$\pi_k:\ell_{r_k}^{N}\to\ell_{r_k}^{N_k}$ by}
{
\[
 \iota_k(x):=(x_1,\ldots,x_{N_k},0,\ldots,0),\qquad
 \pi_k(x):=(x_1,\ldots,x_{N_k}).
\]}

Then $\|\iota_k\|=\|\pi_k\|=1$ and $\pi_k\iota_k$ is the
identity.  Define
\[
 \widetilde B:\ell_{r_1}^{N}\times\cdots\times\ell_{r_s}^{N}\to\C,
 \qquad
 \widetilde B(x^{(1)},\ldots,x^{(s)})
 :=B(\pi_1x^{(1)},\ldots,\pi_sx^{(s)}).
\]
The two inequalities
\[
 \|\widetilde B\|\leq\|B\|,
 \qquad
 |B(x^{(1)},\ldots,x^{(s)})|
 =|\widetilde B(\iota_1x^{(1)},\ldots,\iota_sx^{(s)})|
 \leq\|\widetilde B\|
\]
for $\|x^{(k)}\|_{r_k}\leq1$ show that
$\|\widetilde B\|=\|B\|$.  Moreover,
\begin{equation}\label{eq:extended-B-coefficients}
 \widetilde B(e_{j_1},\ldots,e_{j_s})=
 \begin{cases}
 B(e_{j_1},\ldots,e_{j_s}),&1\leq j_k\leq N_k\ (1\leq k\leq s),\\
 0,&j_k>N_k\text{ for some }k.
 \end{cases}
\end{equation}
Thus it is enough to prove the assertion in the common dimension $N$.

{We use the cited transfer lemma only in the following form: a mixed estimate of the form \eqref{eq:transfer-source} on $(\ell_\infty^N)^s$ transfers, with the same constant $C$, to the corresponding mixed estimate on $\ell_{p_1}^N\times\cdots\times\ell_{p_s}^N$, with outer exponent $\eta_1$, provided $t\ge\eta_1$.}
In the notation of that transfer result, take
\[
 p_k=r_k,\qquad q_k=\infty\quad(1\leq k\leq s),
 \qquad \lambda_0=1.
\]
The exponent denoted by $\eta_1$ there is
\[
 \eta_1=
 \left(\frac1{\lambda_0}-\sum_{k=1}^s\frac1{p_k}
                         +\sum_{k=1}^s\frac1{q_k}\right)^{-1}
 =(1-\sigma)^{-1}=\kappa.
\]
Thus the hypothesis $t\geq\kappa$ in the present lemma is precisely the
condition $t\geq\eta_1$ required there. Applying the hypothesis
\eqref{eq:transfer-source} with $M=N$ gives
\eqref{eq:specialized-transfer} with the same constant $C$.
By \eqref{eq:extended-B-coefficients}, the estimate \eqref{eq:specialized-transfer} reduces to the indices $1\leq j_k\leq N_k$ for $1\leq k\leq s$.
\end{proof}

{
The exponent in the next lemma is the classical multilinear Hardy--Littlewood
exponent.  The regime $\sigma\leq1/2$ goes back to Praciano-Pereira
\cite{PracianoPereira1981}; anisotropic and mixed-norm formulations are given
in \cite{AlbuquerqueBayartPellegrinoSeoane2016}, and the two-regime formula is
recorded explicitly in Ara\'ujo and C\^amara~\cite{AraujoCamara2018}.  We state
the estimate in the form needed below and keep track of the complex constant
through Lemmas~\ref{lem:mixed-source} and~\ref{lem:specialized-transfer}; see
also \cite{CavalcanteNunezPellegrinoRueda2022} for the complex mixed
Littlewood constants.
}

\begin{lemma}\label{lem:anisotropic}
Let $s,N_1,\ldots,N_s\in\N$, let $r_1,\ldots,r_s\in(1,\infty]$,
and suppose that
\[
 \sigma:=\sum_{k=1}^s\frac1{r_k}<1.
\]
Define
\begin{equation}\label{eq:gamma-D}
 \begin{split}
 {\gamma}(s,\sigma)&:=
 \begin{cases}
 \dfrac{2s}{s+1-2\sigma},&0\leq\sigma\leq\dfrac12,\\[2mm]
 \dfrac1{1-\sigma},&\dfrac12\leq\sigma<1,
 \end{cases}\\[1mm]
 {K_\C}(s,\sigma)&:=
 \begin{cases}
 (2/\sqrt\pi)^{s-1},&0\leq\sigma\leq\dfrac12,\\[1mm]
 (2/\sqrt\pi)^{2(s-1)(1-\sigma)},&\dfrac12\leq\sigma<1.
 \end{cases}
 \end{split}
\end{equation}
Every $s$-linear form
$B:\ell_{r_1}^{N_1}\times\cdots\times\ell_{r_s}^{N_s}\to\C$ satisfies
\begin{equation}\label{eq:anisotropic}
 \left(\sum_{j_1=1}^{N_1}\cdots\sum_{j_s=1}^{N_s}
 |B(e_{j_1},\ldots,e_{j_s})|^{{\gamma}(s,\sigma)}\right)^{1/{\gamma}(s,\sigma)}
 \leq {K_\C}(s,\sigma)\|B\|.
\end{equation}
\end{lemma}

\begingroup
\begin{proof}
For $s=1$, \eqref{eq:anisotropic} is the usual duality identity.  Assume
$s\ge2$ and first that all $r_k<\infty$, and put
\[
 \kappa:=(1-\sigma)^{-1}.
\]
Lemmas~\ref{lem:mixed-source} and~\ref{lem:specialized-transfer} give, for
$t\ge\max\{2,\kappa\}$ and every $1\le i\le s$,
\begin{equation}\label{eq:target-mixed}
 \left[\sum_{a=1}^{N_i}
 \left(\sum_{\substack{1\le j_k\le N_k\\ k\ne i}}
 |B(e_{j_1},\ldots,e_{j_{i-1}},e_a,e_{j_{i+1}},\ldots,e_{j_s})|^t\right)^{\kappa/t}
 \right]^{1/\kappa}
 \leq\left(\frac2{\sqrt\pi}\right)^{2(s-1)/t}\|B\|.
\end{equation}

If $0\le\sigma\le1/2$, then $\kappa\le2$, so take $t=2$.  Put
$b_{j_1,\ldots,j_s}:=|B(e_{j_1},\ldots,e_{j_s})|$ and denote the left-hand
side of \eqref{eq:target-mixed} by $M_i$.
\minorfix{Since $\kappa\le2$, Minkowski's inequality moves the
$\ell_\kappa$ norm from the outermost position to the position of its
coordinate in the fixed nesting order $j_1,\ldots,j_s$, without increasing
the mixed norm. Mixed H\"older interpolation of these reordered norms,
with equal weights $1/s$ (see
\cite{AlbuquerqueBayartPellegrinoSeoane2016,BayartPellegrinoSeoane2014}),
gives}
\[
 \left(\sum_{j_1=1}^{N_1}\cdots\sum_{j_s=1}^{N_s}
 b_{j_1,\ldots,j_s}^{\gamma}\right)^{1/\gamma}
 \leq \prod_{i=1}^s M_i^{1/s},
 \qquad
 \frac1\gamma=\frac1s\left(\frac1\kappa+\frac{s-1}{2}\right)
 =\frac{s+1-2\sigma}{2s}.
\]
Since \eqref{eq:target-mixed} with $t=2$ gives
$M_i\le(2/\sqrt\pi)^{s-1}\|B\|$ for every $i$, we obtain
\eqref{eq:anisotropic} with
$K_\C(s,\sigma)=(2/\sqrt\pi)^{s-1}$.

If $1/2\le\sigma<1$, take $t=\kappa$ in
\eqref{eq:target-mixed}.  The left-hand side is then the full
$\ell_\kappa$ coefficient norm, and
\[
 \left(\frac2{\sqrt\pi}\right)^{2(s-1)/\kappa}
 =\left(\frac2{\sqrt\pi}\right)^{2(s-1)(1-\sigma)},
\]
which is exactly the second branch of \eqref{eq:gamma-D}.

Finally, let
\[
 I_\infty:=\{k:r_k=\infty\},\qquad d:=\card I_\infty,
 \qquad
 r_k^{(\nu)}:=\begin{cases}r_k,&k\notin I_\infty,\\ \nu,&k\in I_\infty,\end{cases}
 \qquad
 \sigma_\nu:=\sigma+\frac d\nu.
\]
For $\nu>d/(1-\sigma)$ one has $\sigma_\nu<1$.  Regard $B$ as a form
$B_\nu$ on
$\ell_{r_1^{(\nu)}}^{N_1}\times\cdots\times\ell_{r_s^{(\nu)}}^{N_s}$.
Since $\|x\|_\infty\le\|x\|_\nu$, we have $\|B_\nu\|\le\|B\|$.
Applying the finite-exponent case gives
\[
 \left(\sum_{j_1=1}^{N_1}\cdots\sum_{j_s=1}^{N_s}
 |B(e_{j_1},\ldots,e_{j_s})|^{\gamma(s,\sigma_\nu)}\right)^{1/\gamma(s,\sigma_\nu)}
 \leq K_\C(s,\sigma_\nu)\|B\|.
\]
Letting $\nu\to\infty$ proves \eqref{eq:anisotropic}, since the coefficient
array is finite and both functions in \eqref{eq:gamma-D} are continuous at
$\sigma=1/2$.
\end{proof}
\endgroup

\subsection{Multiplicity-pattern estimates}\label{sec:decomposition}

{A partition of $m$ is a finite sequence $\lambda=(\lambda_1,\ldots,\lambda_s)$ of positive integers satisfying}
\begin{equation}\label{eq:partition-definition}
 \lambda_1\geq\cdots\geq\lambda_s\geq1,
 \qquad
 \sum_{t=1}^s\lambda_t=m.
\end{equation}
We write $\ell(\lambda)=s$. For $1\leq r\leq m$, define
\[
 m_r(\lambda):={\card}
 \{t\in\{1,\ldots,s\}:\lambda_t=r\}.
\]
For a partition $\lambda=(\lambda_1,\ldots,\lambda_s)$ of $m$ and
$n\in\mathbb N$, define
\begin{equation}\label{eq:pattern-set}
 \mathcal A_\lambda(n):=
 \left\{\alpha=(\alpha_1,\ldots,\alpha_n)\in\mathcal M_m(n):
 \begin{array}{l}
 {\card}\{j\in\{1,\ldots,n\}:\alpha_j=r\}=m_r(\lambda)\\[-1mm]
 \text{for every }r\in\{1,\ldots,m\}
 \end{array}
 \right\}.
\end{equation}
Equivalently, $\alpha\in\mathcal A_\lambda(n)$ precisely when the positive
coordinates of $\alpha$, arranged in decreasing order, are
$\lambda_1,\ldots,\lambda_s$.

We write $\lambda\vdash m$ when $\lambda$ is a partition of $m$.

For $m<p\leq\infty$, define
\begin{equation}\label{eq:G-definition}
 G_p(\lambda):=
 \begin{cases}
 \displaystyle\prod_{t=1}^s
      \left(\frac m{\lambda_t}\right)^{\lambda_t/p},&p<\infty,\\[2mm]
 1,&p=\infty.
 \end{cases}
\end{equation}
The projection and multilinear representation build on the construction in \cite[Lemma~2.3 and the proof of Lemma~3.2]{NunezPellegrinoRaposoTeixeira2026}; the refinement used here keeps the coordinate blocks separate.

\begin{lemma}\label{lem:block-projection}
Let $m\geq2$, $m<p\leq\infty$, and let
$\lambda=(\lambda_1,\ldots,\lambda_s)$ be a partition of $m$ into $s$ positive parts. Let
\[
 P:\ell_p^n(\C)\longrightarrow\C,\qquad
 P(z)=\sum_{\alpha\in\mathcal M_m(n)}a_\alpha z^\alpha.
\]
Fix a map
\[
 \chi:\{1,\ldots,n\}\longrightarrow\{1,\ldots,s\},\qquad
 S_t:=\chi^{-1}(\{t\})\quad(1\leq t\leq s).
\]
Define
\begin{equation}\label{eq:block-coefficient-set}
 \Acal_{\lambda,\chi}(n):=
 \left\{\sum_{t=1}^s\lambda_t e_{j_t}:
                       j_t\in S_t\ (1\leq t\leq s)\right\}
 \subseteq\Acal_\lambda(n)
\end{equation}
and the $s$-linear form
\begin{equation}\label{eq:block-spaces}
 B_{\lambda,\chi}:\ell_{p/\lambda_1}(S_1)\times\cdots\times
                    \ell_{p/\lambda_s}(S_s)\longrightarrow\C
\end{equation}
by
\begin{equation}\label{eq:block-form}
 B_{\lambda,\chi}(x^{(1)},\ldots,x^{(s)})
 :=\sum_{j_1\in S_1}\cdots\sum_{j_s\in S_s}
 a_{\lambda_1 e_{j_1}+\cdots+\lambda_s e_{j_s}}
 \prod_{t=1}^s x^{(t)}_{j_t}.
\end{equation}
Here $p/\lambda_t=\infty$ if $p=\infty$. Then
\begin{equation}\label{eq:B-norm-bound}
 \|B_{\lambda,\chi}\|\leq G_p(\lambda)\|P\|_p.
\end{equation}
\end{lemma}

\begin{proof}
If some $S_t$ is empty, the sum in \eqref{eq:block-form} is empty and
$B_{\lambda,\chi}=0$. Suppose every $S_t$ is nonempty.
For $\omega\in\T^s$, define
\[
 V_\omega:\C^n\longrightarrow\C^n,\qquad
 (V_\omega z)_j:=\omega_{\chi(j)}z_j\quad(1\leq j\leq n).
\]
If $p<\infty$, then
\[
 \|V_\omega z\|_p^p
 =\sum_{t=1}^s\sum_{j\in S_t}|\omega_tz_j|^p
 =\sum_{j=1}^n|z_j|^p=\|z\|_p^p.
\]
For $p=\infty$, the corresponding equality is
$\max_{1\leq j\leq n}|\omega_{\chi(j)}z_j|
 =\max_{1\leq j\leq n}|z_j|$.
{Define the linear operator
$\Pi_{\lambda,\chi}:\Pol_m(\C^n)\to\Pol_m(\C^n)$ by}
\begin{equation}\label{eq:block-phase-projection}
 (\Pi_{\lambda,\chi}P)(z)
 :=\int_{\T^s}P(V_\omega z)
        \prod_{t=1}^s\overline{\omega_t}^{\lambda_t}\,d\mathfrak m_s(\omega).
\end{equation}
For $\alpha\in\mathcal M_m(n)$, put
$d_t(\alpha):=\sum_{j\in S_t}\alpha_j$. The standard orthogonality of the characters of $\mathbb T$ gives
{\small
\begin{equation}\label{eq:block-phase-monomial}
 \Pi_{\lambda,\chi}(z^\alpha)
 =z^\alpha\prod_{t=1}^s\int_{\T}\omega_t^{d_t(\alpha)-\lambda_t}\,d\mathfrak m_1(\omega_t)
 =\begin{cases}
 z^\alpha,&d_t(\alpha)=\lambda_t\ (1\leq t\leq s),\\
 0,&\text{otherwise}.
 \end{cases}
\end{equation}}
For $\|z\|_p\leq1$,
\[
 |(\Pi_{\lambda,\chi}P)(z)|
 \leq\int_{\T^s}|P(V_\omega z)|\,d\mathfrak m_s(\omega)
 \leq\|P\|_p\,\mathfrak m_s(\T^s)=\|P\|_p.
\]
Taking the supremum over $\|z\|_p\leq1$ therefore gives
\begin{equation}\label{eq:phase-contractive}
 \|\Pi_{\lambda,\chi}P\|_p\leq\|P\|_p.
\end{equation}

The projection $\Pi_{\lambda,\chi}$ fixes the total degree $\lambda_t$ in each
block $S_t$. To retain only exponents divisible by $\lambda_t$ inside each block,
for $j\in\{1,\ldots,n\}$, set
$\nu_j:=\lambda_{\chi(j)}$ and $\xi_j:=\exp(2\pi i/\nu_j)$.
For $k\in\{0,\ldots,\nu_j-1\}$, define
\[
 U_{j,k}:\C^n\to\C^n,\qquad
 U_{j,k}(z):=z+(\xi_j^k-1)z_je_j.
\]
Define $R_j:\Pol_m(\C^n)\to\Pol_m(\C^n)$ by
\begin{equation}\label{eq:root-projection}
 (R_jQ)(z):=\frac1{\nu_j}\sum_{k=0}^{\nu_j-1}Q(U_{j,k}z).
\end{equation}
Since $|\xi_j|=1$ and $U_{j,k}$ maps the $\ell_p$ unit ball onto itself,
{
\begin{equation}\label{eq:root-contractive}
 \|R_jQ\|_p\leq\frac1{\nu_j}\sum_{k=0}^{\nu_j-1}\|Q\circ U_{j,k}\|_p
 =\frac1{\nu_j}\sum_{k=0}^{\nu_j-1}\|Q\|_p=\|Q\|_p.
\end{equation}
}
For $\alpha\in\mathcal M_m(n)$,
\[
 R_j(z^\alpha)=\frac1{\nu_j}\left(\sum_{k=0}^{\nu_j-1}
                              \xi_j^{k\alpha_j}\right)z^\alpha.
\]
If $\nu_j\mid\alpha_j$, the sum is $\nu_j$. Otherwise
$\xi_j^{\alpha_j}\neq1$ and
\[
 \sum_{k=0}^{\nu_j-1}\xi_j^{k\alpha_j}
 =\frac{1-(\xi_j^{\alpha_j})^{\nu_j}}{1-\xi_j^{\alpha_j}}=0.
\]
Consequently,
\begin{equation}\label{eq:root-monomial}
 R_j(z^\alpha)=
 \begin{cases}z^\alpha,&\nu_j\mid\alpha_j,\\0,&\nu_j\nmid\alpha_j.
 \end{cases}
\end{equation}
Let
\[
 \mathcal Q_{\lambda,\chi}:=R_n\circ\cdots\circ R_1\circ\Pi_{\lambda,\chi}
 :\Pol_m(\C^n)\longrightarrow\Pol_m(\C^n).
\]
By \eqref{eq:block-phase-monomial} and \eqref{eq:root-monomial}, a
monomial is unchanged by $\mathcal Q_{\lambda,\chi}$ exactly when
\[
 \sum_{j\in S_t}\alpha_j=\lambda_t,
 \qquad \alpha_j\in\lambda_t\N_0\quad(j\in S_t),
 \qquad 1\leq t\leq s.
\]
Indeed, writing $\alpha_j=\lambda_t b_j$ within $S_t$ gives
$b_j\in\N_0$ and $\sum_{j\in S_t}b_j=1$. Exactly one $b_j$ equals
$1$ and the others are zero. Thus
{\small
\begin{equation}\label{eq:projected-polynomial}
 (\mathcal Q_{\lambda,\chi}P)(z)
 =\sum_{\alpha\in\Acal_{\lambda,\chi}(n)}a_\alpha z^\alpha
 =B_{\lambda,\chi}\bigl((z_j^{\lambda_1})_{j\in S_1},\ldots,(z_j^{\lambda_s})_{j\in S_s}\bigr).
\end{equation}}
This coefficient formula also gives
$\mathcal Q_{\lambda,\chi}^2=\mathcal Q_{\lambda,\chi}$.
Equations \eqref{eq:phase-contractive} and \eqref{eq:root-contractive}
imply
\begin{equation}\label{eq:projection-norm}
 \|\mathcal Q_{\lambda,\chi}P\|_p\leq\|P\|_p.
\end{equation}

Take $x^{(t)}\in\ell_{p/\lambda_t}(S_t)$ with
$\|x^{(t)}\|_{p/\lambda_t}\leq1$ for $1\leq t\leq s$.
For each $j\in S_t$, choose $y_j\in\C$ satisfying
$y_j^{\lambda_t}=x_j^{(t)}$. Thus
$|y_j|=|x_j^{(t)}|^{1/\lambda_t}$.
If $p<\infty$, set
\[
 c_t:=\left(\frac{\lambda_t}{m}\right)^{1/p},\qquad
 w:=(c_{\chi(j)}y_j)_{j=1}^n\in\C^n.
\]
{By the definitions of $c_t$ and $y_j$,
\[
 \|w\|_p^p=\sum_{t=1}^s\frac{\lambda_t}{m}
 \sum_{j\in S_t}|x_j^{(t)}|^{p/\lambda_t}
 \leq\sum_{t=1}^s\frac{\lambda_t}{m}=1.
\]}
Using \eqref{eq:projected-polynomial} at $w$ gives
\[
 (\mathcal Q_{\lambda,\chi}P)(w)
 =\left(\prod_{t=1}^s c_t^{\lambda_t}\right)
                     B_{\lambda,\chi}(x^{(1)},\ldots,x^{(s)}).
\]
{Combining this identity with \eqref{eq:projection-norm} gives
\[
 |B_{\lambda,\chi}(x^{(1)},\ldots,x^{(s)})|
 \leq\left(\prod_{t=1}^s c_t^{-\lambda_t}\right)\|\mathcal Q_{\lambda,\chi}P\|_p
 \leq G_p(\lambda)\|P\|_p.
\]}
For $p=\infty$, take $c_t=1$ for every $t$ and $w=(y_j)_{j=1}^n$.
Then $\|w\|_\infty\leq1$, and \eqref{eq:projected-polynomial} gives
$|B_{\lambda,\chi}(x^{(1)},\ldots,x^{(s)})|\leq\|P\|_\infty$.
Taking the supremum over the $x^{(t)}$ proves \eqref{eq:B-norm-bound}.
\end{proof}

{
For $n,s\in\N$, let
\begin{equation}\label{eq:color-space}
 \Omega_{n,s}:=\{\chi:\{1,\ldots,n\}\to\{1,\ldots,s\}\}.
\end{equation}
For a partition $\lambda=(\lambda_1,\ldots,\lambda_s)$ of $m$ of length $s$, set
\[
 \mathfrak p_\lambda:=\frac1{s^s}\prod_{r=1}^m m_r(\lambda)!.
\]
For every $\alpha\in\Acal_\lambda(n)$, the fraction of maps $\chi\in\Omega_{n,s}$
for which $\alpha\in\Acal_{\lambda,\chi}(n)$ is
\begin{equation}\label{eq:survival-probability}
 \frac1{s^n}\card\{\chi\in\Omega_{n,s}:\alpha\in\Acal_{\lambda,\chi}(n)\}
 =\mathfrak p_\lambda.
\end{equation}
}
Indeed, for $1\leq r\leq m$, define
\[
 J_r(\alpha):=\{j\in\{1,\ldots,n\}:\alpha_j=r\},\qquad
 T_r(\lambda):=\{t\in\{1,\ldots,s\}:\lambda_t=r\}.
\]
Both sets have cardinality $m_r(\lambda)$. By
\eqref{eq:block-coefficient-set}, the condition
$\alpha\in\Acal_{\lambda,\chi}(n)$ is equivalent to requiring
$\chi|_{J_r(\alpha)}:J_r(\alpha)\to T_r(\lambda)$ to be a bijection
for each $r$. There are $m_r(\lambda)!$ choices for this restriction.
The $n-s$ coordinates for which $\alpha_j=0$ may be mapped arbitrarily.
Thus
\[
 \card\{\chi\in\Omega_{n,s}:\alpha\in\Acal_{\lambda,\chi}(n)\}
 =s^{n-s}\prod_{r=1}^m m_r(\lambda)!.
\]
Dividing by $s^n$ proves \eqref{eq:survival-probability}.

Let $G_p(\lambda)$ be defined by \eqref{eq:G-definition}.
Set $q:=q(m,p)$, with $q(m,p)$ defined in
\eqref{eq:HL-exponent}, and put
\begin{equation}\label{eq:D-pattern}
 {K_{s,m,p}:=K_\C\!\left(s,\frac mp\right)}
 =
 \begin{cases}
 (2/\sqrt\pi)^{2(s-1)/q},&m<p<2m,\\
 (2/\sqrt\pi)^{s-1},&2m\leq p\leq\infty.
 \end{cases}
\end{equation}
\begin{lemma}\label{lem:one-pattern}
Let $m\geq2$, $m<p\leq\infty$, and let
$\lambda=(\lambda_1,\ldots,\lambda_s)$ be a partition of $m$ into $s$ positive parts. Every
$P:\ell_p^n(\C)\to\C$, $P(z)=\sum_{\alpha\in\mathcal M_m(n)}a_\alpha z^\alpha$,
satisfies
\begin{equation}\label{eq:one-pattern}
 \left(\sum_{\alpha\in\Acal_\lambda(n)}|a_\alpha|^q\right)^{1/q}
 \leq {K_{s,m,p}}G_p(\lambda)\mathfrak p_\lambda^{-1/q}\|P\|_p,
 \qquad q=q(m,p).
\end{equation}
\end{lemma}

\begin{proof}
If $s>n$, then $\Acal_\lambda(n)=\varnothing$ and the assertion holds.
Suppose $s\leq n$. Fix $\chi\in\Omega_{n,s}$. If some $S_t$ is empty,
$\Acal_{\lambda,\chi}(n)=\varnothing$. Otherwise apply
Lemma~\ref{lem:anisotropic} to the form in \eqref{eq:block-form} with
\[
 r_t=\frac p{\lambda_t}>1\quad(1\leq t\leq s),\qquad
 \sum_{t=1}^s\frac1{r_t}=\frac1p\sum_{t=1}^s\lambda_t=\frac mp<1.
\]
The convention $1/\infty=0$ includes $p=\infty$.
{If $m<p<2m$, then $m/p>1/2$, so the second branch of
\eqref{eq:gamma-D} applies. Since
\[
 1-\frac mp=\frac{p-m}{p}=\frac1q,
\]
we obtain
\[
 \gamma(s,m/p)=q,
 \qquad
 K_\C(s,m/p)=\left(\frac2{\sqrt\pi}\right)^{2(s-1)/q}=K_{s,m,p}.
\]}

If $2m\leq p\leq\infty$, then
\begin{equation}\label{eq:pattern-exponent-comparison}
 \frac1{{\gamma}(s,m/p)}
 =\frac12+\frac{1-2m/p}{2s}
 \geq\frac12+\frac{1-2m/p}{2m}
 =\frac1q,
\end{equation}
since $s\leq m$ and $1-2m/p\geq0$. Thus ${\gamma}(s,m/p)\leq q$;
monotonicity of finite sequence norms gives the $\ell_q$ estimate
with constant ${K_{s,m,p}}=(2/\sqrt\pi)^{s-1}$.
Since the sets $S_t$ are disjoint, each $\alpha\in\Acal_{\lambda,\chi}(n)$
has a unique expression $\alpha=\sum_{t=1}^s\lambda_t e_{j_t}$ with
$j_t\in S_t$. Therefore
\[
 \sum_{j_1\in S_1}\cdots\sum_{j_s\in S_s}
 |B_{\lambda,\chi}(e_{j_1},\ldots,e_{j_s})|^q
 =\sum_{\alpha\in\Acal_{\lambda,\chi}(n)}|a_\alpha|^q.
\]
In both ranges, \eqref{eq:B-norm-bound} now gives
\begin{equation}\label{eq:projected-q}
 \sum_{\alpha\in\Acal_{\lambda,\chi}(n)}|a_\alpha|^q
 \leq {K_{s,m,p}}^qG_p(\lambda)^q\|P\|_p^q.
\end{equation}
If some $S_t$ is empty, then $\Acal_{\lambda,\chi}(n)=\varnothing$, so \eqref{eq:projected-q} remains valid.

{
By \eqref{eq:survival-probability} and \eqref{eq:projected-q},
\[
 \mathfrak p_\lambda\sum_{\alpha\in\Acal_\lambda(n)}|a_\alpha|^q
 =\frac1{s^n}\sum_{\chi\in\Omega_{n,s}}\sum_{\alpha\in\Acal_{\lambda,\chi}(n)}|a_\alpha|^q
 \leq {K_{s,m,p}}^qG_p(\lambda)^q\|P\|_p^q.
\]
}
Since $\mathfrak p_\lambda>0$, division and taking $q$th roots prove
\eqref{eq:one-pattern}.
\end{proof}

\subsection{Summation over multiplicity patterns}\label{sec:global}

{Before summing the estimates of Lemma~\ref{lem:one-pattern}, we record two elementary ingredients: a bound for $G_p(\lambda)$ and a counting identity for partitions of fixed length.}

Set $q=q(m,p)$ and
\begin{equation}\label{eq:theta-definition}
 \theta:=\frac{mq}{p},
\end{equation}
with $\theta=0$ for $p=\infty$. The quantity
$G_p(\lambda)$ is defined in \eqref{eq:G-definition}.

For every partition $\lambda=(\lambda_1,\ldots,\lambda_s)$ of $m$ of length $s$,
\begin{equation}\label{eq:G-bound}
 G_p(\lambda)^q\leq s^\theta.
\end{equation}
Moreover,
\[
 \theta=
 \begin{cases}
 m/(p-m)>1,&m<p<2m,\\[1mm]
 2m^2/[p(m+1)-2m]\leq1,&2m\leq p<\infty,\\[1mm]
 0,&p=\infty.
 \end{cases}
\]
For $p=\infty$, both sides of \eqref{eq:G-bound} are $1$.
Let $p<\infty$ and set $x_t:=\lambda_t/m$ for $1\leq t\leq s$.
Then $x_t>0$ and $\sum_{t=1}^s x_t=1$. {Concavity of the logarithm gives directly
\[
 \sum_{t=1}^s x_t\log\frac1{x_t}\leq\log\left(\sum_{t=1}^s x_t\frac1{x_t}\right)=\log s.
\]}
{This is also the standard maximum-entropy bound on $s$ points; see Cover and Thomas \cite[Theorem~2.6.4]{CoverThomas2006}.}
Therefore
{
\[
 \log G_p(\lambda)^q
 =\frac qp\sum_{t=1}^s\lambda_t\log\frac m{\lambda_t}
 =\frac{mq}{p}\sum_{t=1}^s x_t\log\frac1{x_t}
 \leq\theta\log s.
\]
}
Exponentiation proves \eqref{eq:G-bound}. If $m<p<2m$, then
$mq/p=m/(p-m)>1$. For $2m\leq p<\infty$,
\[
 \frac{mq}{p}=\frac{2m^2}{p(m+1)-2m}\leq1,
 \qquad p(m+1)-2m\geq2m^2.
\]
This proves the stated formulas for $\theta$ and completes the estimate.

Let $\lambda\vdash m$, with ${\ell}(\lambda)=s$ and multiplicities
$m_r(\lambda)$ defined by
$m_r(\lambda)=\card\{t:\lambda_t=r\}$ as in Section~\ref{sec:decomposition}.
The sum in \eqref{eq:composition-identity} is over partitions of $m$ of length $s$.

{For $1\leq s\leq m$, the standard count of compositions gives
\[
 \card\{(k_1,\ldots,k_s)\in\N^s:k_1+\cdots+k_s=m\}
 =\binom{m-1}{s-1};
\]
see Stanley~\cite[p.~18]{Stanley2012}.  A partition $\lambda\vdash m$ of
length $s$ has exactly
\[
 \frac{s!}{\prod_{r=1}^m m_r(\lambda)!}
\]
distinct orderings.  Grouping the compositions according to their decreasing
rearrangement therefore gives
\[
 \binom{m-1}{s-1}
 =\sum_{\substack{\lambda\vdash m\\\ell(\lambda)=s}}
 \frac{s!}{\prod_{r=1}^m m_r(\lambda)!}.
\]
Dividing by $s!$ yields
\begin{equation}\label{eq:composition-identity}
 \sum_{\substack{\lambda\vdash m\\\ell(\lambda)=s}}
 \frac1{\prod_{r=1}^m m_r(\lambda)!}
 =\frac1{s!}\binom{m-1}{s-1}.
\end{equation}}

\begin{lemma}\label{lem:global-sum}
For $m\geq2$ and $m<p\leq\infty$,
\begin{equation}\label{eq:global-sum}
 ({D_{m,p}})^q
 \leq\sum_{s=1}^m s^\theta {K_{s,m,p}}^{\,q}
                      \frac{s^s}{s!}\binom{m-1}{s-1}.
\end{equation}
\end{lemma}
\begin{proof}
Let $P\in\Pol_m(\C^n)$. The classes in \eqref{eq:pattern-set}
partition $\mathcal M_m(n)$. Hence
\begin{align*}
 |P|_q^q
 &=\sum_{s=1}^m\sum_{\substack{\lambda\vdash m\\{\ell}(\lambda)=s}}
                     \sum_{\alpha\in\Acal_\lambda(n)}|a_\alpha|^q\\
 &{\overset{\text{Lemma~\ref{lem:one-pattern}}}{\leq}}
 \|P\|_p^q\sum_{s=1}^m
 {K_{s,m,p}}^q\sum_{\substack{\lambda\vdash m\\{\ell}(\lambda)=s}}
       G_p(\lambda)^q\frac{s^s}{\prod_{r=1}^m m_r(\lambda)!}\\
 &{\overset{\eqref{eq:G-bound}}{\leq}}
 \|P\|_p^q\sum_{s=1}^m
 s^\theta {K_{s,m,p}}^q s^s
       \sum_{\substack{\lambda\vdash m\\{\ell}(\lambda)=s}}
                       \frac1{\prod_{r=1}^m m_r(\lambda)!}\\
 &{\overset{\eqref{eq:composition-identity}}{=}}
 \|P\|_p^q\sum_{s=1}^m
 s^\theta {K_{s,m,p}}^q\frac{s^s}{s!}\binom{m-1}{s-1}.
\end{align*}
The bound holds for every $n$ and every $P$, which proves
\eqref{eq:global-sum}.
\end{proof}

\begin{proposition}\label{prop:pattern-global}
For every $m\geq2$ and $m<p\leq\infty$,
\begin{equation}\label{eq:pattern-upper}
 {D_{m,p}}
 \leq m^{m/p}
 \left(1+\frac{4\e}{\pi}\right)^{(m-1)/q(m,p)},
\end{equation}
and
\begin{equation}\label{eq:pattern-lower}
 {D_{m,p}}\geq m^{m/p}.
\end{equation}
\end{proposition}

\begin{proof}
We first prove the upper bound \eqref{eq:pattern-upper}.  It is enough to show
\begin{equation}\label{eq:pattern-upper-q}
 ({D_{m,p}})^q
 \leq m^\theta\left(1+\frac{4\e}{\pi}\right)^{m-1},
\end{equation}
because $\theta/q=m/p$.

By \eqref{eq:D-pattern},
\begin{equation}\label{eq:Dq-bound}
 {K_{s,m,p}}^q\leq(4/\pi)^{s-1}.
\end{equation}
For $m<p<2m$, this follows directly from the definition of $K_{s,m,p}$; for
$2m\leq p\leq\infty$, since $q\leq2$ and $2/\sqrt\pi>1$,
\[
 K_{s,m,p}^q=(2/\sqrt\pi)^{q(s-1)}
 \leq(2/\sqrt\pi)^{2(s-1)}=(4/\pi)^{s-1}.
\]
{
We shall use the elementary estimate
\begin{equation}\label{eq:factorial-bound}
 \frac{s^s}{s!}\leq\e^{s-1},
\end{equation}
which follows directly from
$\log(s!)\geq\int_1^s\log x\,dx=s\log s-s+1$.
Using \eqref{eq:Dq-bound} and \eqref{eq:factorial-bound} in
\eqref{eq:global-sum}, and then using $s\leq m$ together with
$\theta\geq0$, we obtain
\begin{equation}\label{eq:pattern-summed}
\begin{aligned}
 ({D_{m,p}})^q
 &\leq \sum_{s=1}^m s^\theta
       \left(\frac4\pi\right)^{s-1}\e^{s-1}
       \binom{m-1}{s-1}\\
 &\leq m^\theta\sum_{s=1}^m
       \binom{m-1}{s-1}\left(\frac{4\e}{\pi}\right)^{s-1}\\
 &=m^\theta\left(1+\frac{4\e}{\pi}\right)^{m-1}.
\end{aligned}
\end{equation}
}
Thus \eqref{eq:pattern-upper-q} holds. Taking $q$th roots in
\eqref{eq:pattern-upper-q} and using $\theta/q=m/p$ proves the upper bound
\eqref{eq:pattern-upper}.

We now prove the lower bound \eqref{eq:pattern-lower}. Take
$P:\ell_p^m(\C)\to\C$, $P(z):=z_1\cdots z_m$. Its coefficient norm is
$|P|_q=1$. If $p<\infty$ and $\sum_{j=1}^m|z_j|^p\leq1$, the
arithmetic--geometric mean inequality gives
\[
 |P(z)|^p=\prod_{j=1}^m|z_j|^p
 \leq\left(\frac1m\sum_{j=1}^m|z_j|^p\right)^m\leq m^{-m}.
\]
Equality holds at $z_j=m^{-1/p}$ for every $j$, so
$\|P\|_p=m^{-m/p}$. For $p=\infty$, $\|P\|_\infty=1$. Hence
${D_{m,p}}\geq m^{m/p}$, which is \eqref{eq:pattern-lower}.
\end{proof}

\section{Proof of Theorem C: the Hardy--Littlewood endpoint and stability at \texorpdfstring{$p=m$}{p=m}}
\label{sec:proof-C}

{At the boundary $p=m$, the Hardy--Littlewood coefficient exponent is $\infty$. We prove directly that ${D_{m,m}}=m$.}

\subsection{Proof of Theorem C\textup{(i)}: the exact endpoint \texorpdfstring{$p=m$}{p=m}}

At $p=m$, one has $q(m,m)=\infty$ and
\[
 {D_{m,m}}=m.
\]
Indeed, let
\[
 P:\C^n\longrightarrow\C,\qquad
 P(z)=\sum_{|\alpha|=m}a_\alpha z^\alpha.
\]
{Fix $|\alpha|=m$, put $S:=\operatorname{supp}\alpha$, and set
\[
 r_j:=\left(\frac{\alpha_j}{m}\right)^{1/m}
 \qquad(j\in S).
\]
Since
\[
 \sum_{j\in S}r_j^m=\frac1m\sum_{j\in S}\alpha_j=1,
\]
for every $\zeta=(\zeta_j)_{j\in S}\in\T^S$ the vector
$\sum_{j\in S}r_j\zeta_j e_j$ belongs to the unit sphere of $\ell_m^n$.
By the character-orthogonality argument preceding
\eqref{eq:coefficient-extraction},
\[
 a_\alpha\prod_{j\in S}r_j^{\alpha_j}
 =\int_{\T^S}
 P\!\left(\sum_{j\in S}r_j\zeta_j e_j\right)
 \prod_{j\in S}\overline{\zeta_j}^{\,\alpha_j}
 \,d\zeta,
\]
where the integral is taken with respect to normalized Haar measure on $\T^S$.
Hence
\[
 |a_\alpha|
 \le \|P\|_m
 \prod_{j\in S}\left(\frac{m}{\alpha_j}\right)^{\alpha_j/m}
 \le m\|P\|_m,
\]
because $\alpha_j\ge1$ for $j\in S$ and
$\sum_{j\in S}\alpha_j/m=1$.  Therefore
\[
 \max_{|\alpha|=m}|a_\alpha|\le m\|P\|_m,
 \qquad
 {D_{m,m}}\le m.
\]}

Conversely, let
\[
 P_m:\C^m\longrightarrow\C,\qquad
 P_m(z)=z_1\cdots z_m.
\]
Its largest coefficient has modulus $1$. On the unit ball of $\ell_m^m$,
the arithmetic--geometric mean inequality gives
\[
 |z_1\cdots z_m|
 \le \frac{|z_1|^m+\cdots+|z_m|^m}{m}
 \le \frac1m.
\]
{Equality occurs when $|z_1|=\cdots=|z_m|=m^{-1/m}$, so
\[
 \|P_m\|_m=\frac1m,\qquad D_{m,m}\ge m.
\]}
Combining the two inequalities proves Theorem~\ref{thm:C}\textup{(i)}.

\subsection{Proof of Theorem C(ii)}\label{sec:near-critical}

Set
\[
 A_0:=1+\frac{4\e}{\pi},\qquad L_0:=\log A_0.
\]
For $0<h<m$, the bounds \eqref{eq:pattern-upper} and \eqref{eq:pattern-lower}, with $p=m+h$ and
$q(m,m+h)^{-1}=h/(m+h)$, give
\begin{equation}\label{eq:C-direct-two-sided}
 m^{m/(m+h)}
 \le D_{m,m+h}
 \le m^{m/(m+h)}
 \exp\!\left(\frac{h(m-1)}{m+h}\,\log A_0\right).
\end{equation}
For $h=0$, the same two-sided estimate holds with equality because
$D_{m,m}=m$. Hence \eqref{eq:C-direct-two-sided} is valid for
$0\le h<m$ if the exponential factor is interpreted as $1$ at $h=0$.
Moreover,
\[
 0\le
 \frac{h(m-1)}{m+h}\,\log A_0
 \le L_0h.
\]
Dividing \eqref{eq:C-direct-two-sided} by $m$ and taking logarithms therefore yields
\begin{equation}\label{eq:C-direct-ratio}
 -\frac{h\log m}{m+h}
 \le
 \log\frac{D_{m,m+h}}{m}
 \le
 -\frac{h\log m}{m+h}+L_0h,
 \qquad 0\le h<m.
\end{equation}

Now let $h_m:=p_m-m\to0$. For all sufficiently large $m$ one has
$0\le h_m<m$, and both outer terms in \eqref{eq:C-direct-ratio} tend to $0$.
Thus
\[
 \log\frac{D_{m,p_m}}m\longrightarrow0,
 \qquad\text{hence}\qquad
 D_{m,p_m}=(1+o(1))m.
\]
This proves Theorem~\ref{thm:C}\textup{(ii)}.

\subsection{Proof of Theorem C(iii)}

Let
\[
 h_m:=p_m-m=o\!\left((m/\log m)^{1/3}\right).
\]
The lower bound in \eqref{eq:C-direct-two-sided} gives, for all sufficiently large $m$,
\[
 \frac{\log D_{m,p_m}}{\log m}
 \ge \frac{m}{m+h_m}\longrightarrow1.
\]
To obtain the reverse inequality, fix $a>0$. Since
$h_m=o((m/\log m)^{1/3})$, one eventually has
$0\le h_m\le a(m/\log m)^{1/3}$. If $h_m=0$, the endpoint identity gives
$\log D_{m,p_m}/\log m=1$. If $h_m>0$, Theorem~\ref{thm:B}\textup{(i)} applies. Hence, uniformly over all sufficiently large $m$,
{
\[
 \frac{\log D_{m,p_m}}{\log m}
 \le 1+\frac{8}{3\pi\sqrt3}a^{3/2}+o(1),
\]
}
where the $o(1)$ is needed only on the indices with $h_m>0$. Therefore
\[
 \limsup_{m\to\infty}
 \frac{\log D_{m,p_m}}{\log m}
 \le 1+\frac{8}{3\pi\sqrt3}a^{3/2}.
\]
Letting $a\downarrow0$ and combining with the lower bound gives
\[
 \frac{\log D_{m,p_m}}{\log m}\longrightarrow1,
 \qquad\text{that is,}\qquad
 D_{m,p_m}=m^{1+o(1)}.
\]
This proves Theorem~\ref{thm:C}\textup{(iii)}.

\section*{Acknowledgments}

\paragraph{\bf Funding} 

D. Pellegrino is supported by Grants No.~406457/2023-9
(CNPq/MCTI N\textsuperscript{o}~10/2023) and No.~403964/2024-5
(MCTI/CNPq N\textsuperscript{o}~16/2024). He is also supported by Grant
No.~305807/2025-0 from the Conselho Nacional de Desenvolvimento
Cient\'ifico e Tecnol\'ogico (CNPq, Brazil).
E. Teixeira gratefully acknowledges support from the
Grayce B. Kerr Chair at Oklahoma State University.

This work was conducted in part within the DARPA ExpMath project
\emph{A Human-Centered Framework for AI-Mathematician Collaboration
in Research-Level Mathematics}
(Agreement No.~HR0011262E029).
The views and conclusions expressed here are those of the authors
and should not be interpreted as representing the official policies
of the Department of Defense or the U.S.\ Government.

\bigskip

\paragraph{\bf Competing interests}
The authors declare no competing interests.

\paragraph{\bf Data availability}
No datasets were generated or analyzed in the course of this research.

\paragraph{\bf \auditfix{Declaration of generative AI and AI-assisted technologies}}

The authors conceived the research program, formulated its central ideas, and directed the successive development of its statements and arguments. ChatGPT 5.6 Sol (OpenAI) was used as a collaborative research tool to stress-test ideas and support exploratory analysis.  It also assisted with consistency checks, the organization of arguments, literature searches, and routine \LaTeX{} typesetting\auditfix{, including the TikZ layout of the proof architecture}.

The authors retain intellectual authorship of the work. They wrote and revised the manuscript with {the assistance described in this declaration} and independently verified all mathematical statements and proofs. They reviewed the references and take full responsibility for the correctness, originality, and integrity of the work.

\end{document}